\documentclass[11pt]{amsart}
\usepackage[left=10mm,right=10mm]{geometry}
\usepackage[english]{babel}
\usepackage[T1]{fontenc}
\usepackage[utf8]{inputenc}
\usepackage{amsthm}
\usepackage{amsmath}
\title{Fixation for $\mathcal{U}$-Ising and $\mathcal{U}$-voter dynamics with frozen vertices}
\author{Laure Mar{\^e}ch\'e}
\email{laure.mareche@math.unistra.fr}
\address{Institut de Recherche Mathématique Avancée, 
UMR 7501 Université de Strasbourg et CNRS, 
7 rue René Descartes, 67000 Strasbourg, France}
\theoremstyle{plain}
\newtheorem{theorem}{Theorem}

\newtheorem{lemma}[theorem]{Lemma}
\newtheorem{proposition}[theorem]{Proposition}

\newtheorem{claim}[theorem]{Claim}
\theoremstyle{definition}
\newtheorem{definition}[theorem]{Definition}
\theoremstyle{remark}
\newtheorem{remark}[theorem]{Remark}
\usepackage{dsfont}
\usepackage{hyperref}
\usepackage{xcolor}
\usepackage{tikz}
\usetikzlibrary{patterns}
\usetikzlibrary{patterns.meta}

\begin{document}

\begin{abstract}
The zero-temperature stochastic Ising model is a special case of the famous stochastic Ising model of statistical mechanics, and the voter model is another classical model in this field. In both models, each vertex of the graph $\mathds{Z}^d$ can have one of two states, and can change state to match the state of its neighbors. In 2017, Morris \cite{Morris2017} proposed generalizations of these models, the $\mathcal{U}$-Ising and $\mathcal{U}$-voter dynamics, in which a vertex can change state to match the state of certain subsets of vertices near it. These generalizations were inspired by similar generalizations in the related model of bootstrap percolation, where Balister, Bollob\'as, Duminil-Copin, Morris, Przykucki, Smith and Uzzell \cite{Balister_et_al2016,Bollobas_et_al2017,Bollobas_et_al2015} were able to establish a very impressive universality classification of the generalized models. However, there have been very few results on the $\mathcal{U}$-Ising and $\mathcal{U}$-voter dynamics. The only one is due to Blanquicett \cite{Blanquicett2021}, who obtained a few encouraging advances on the important question of \emph{fixation}, which is only partially solved for the zero-temperature stochastic Ising model: will all vertices eventually settle on a given state or will they oscillate forever between the two states? In this work, we tackle a question which was solved for the zero-temperature stochastic Ising model by Damron, Eckner, Kogan, Newman and Sidoravicius \cite{Damron_et_al2014}: fixation when a fraction of the vertices of $\mathds{Z}^d$ are frozen in one of the states. For $d=1$ and $2$, in most cases we prove that if all frozen vertices are in the same state, all vertices eventually settle at this state. Moreover, if vertices can be frozen in both states but the proportion of vertices frozen in the second state is small enough, we were able to establish a universality classification identifying the models in which some vertices oscillate forever and those in which all vertices eventually settle in one state. 
\end{abstract}

\maketitle


\section{Introduction}

The \emph{voter model} is a statistical mechanics model introduced by Holley and Liggett in \cite{Holley_et_al1975}, which can represent the evolution of the opinion of a population of individuals which are influenced by their neighbors, and tend to change their opinion to agree with them. In the most basic form of the voter model, individuals are represented by vertices of the graph $\mathds{Z}^d$, called \emph{sites}. Each site can have one of two opinions, called \emph{states}, which we denote by $+$ and $-$, and each site, at rate 1, randomly chooses the state of one of its $2d$ neighbors and adopts it. This model is simple, yet interesting, and it received a lot of attention (see \cite{Liggett1999}, especially the Notes and References part).

A closely related model is the \emph{zero-temperature stochastic Ising model}, a special case of the famous stochastic Ising model, which describes the behavior of a magnetic material. In the zero-temperature stochastic Ising model, each vertex of $\mathds{Z}^d$ can be either at state $+$ or at state $-$, and tries to agree with its neighbors, but the dynamics is different. A site will change state at rate 1 if more than $d$ of its neighbors have the opposite state, at rate $1/2$ if exactly $d$ of its neighbors have the opposite state, and will never change state otherwise. This model was also studied a lot (see the review in Section 5 of \cite{Morris2017}, as well as \cite{Lacoin_et_al2013} and the references within it).

One of the most important questions about the zero-temperature stochastic Ising model is that of fixation. For a given site, is there a time after which it remains forever in state $+$ (or $-$), in which case we say it \emph{fixates at $+$} (respectively at $-$) at this time? Or is there always a later time at which the site changes state, in which case it is called a \emph{flipper}? If all vertices are initially independently at $+$ with some probability $p \in (0,1)$, in dimension 1 the results of Arratia \cite{Arratia1983} prove that almost surely all sites are flippers. However, in higher dimension, the problem is harder. It is conjectured that for $d \geq 2$, if $p=1/2$ almost surely all sites are flippers, and if $p>1/2$ almost surely all sites fixate at $+$ (since the model is symmetric, this would imply that for $p<1/2$ almost surely all sites fixate at $-$), but this conjecture is far from proven. Almost sure fixation at $+$ for all sites was only proven for $p$ close to 1, by Fontes, Schonmann and Sidoravicius \cite{Fontes_et_al2002}, and the fact that almost surely all sites are flippers for $p=1/2$ was proven only in dimension 2, by Nanda, Newman and Stein \cite{Nanda_et_al2000}. 

In order to better understand the zero-temperature stochastic Ising model, Chalupa, Leath and Reich \cite{Chalupa_et_al_1979} introduced another model they called \emph{bootstrap percolation}. In bootstrap percolation, any site of $\mathds{Z}^d$ can be either \emph{healthy} or \emph{infected}, infected sites ramain infected, and a healthy site becomes infected when it has at least a given number $r$ of infected neighbors. Bootstrap percolation is very helpful in the study of the zero-temperature stochastic Ising model, because it can be seen that the vertices that can be set at $+$ in the latter are those that can be infected by the former when $r=d$ and the initial infected sites are the sites initially at $+$. Bootstrap percolation turned out to be very interesting in itself, and a lot of work has been done on it (see Sections 2 to 4 of \cite{Morris2017} for a review).

A natural generalization of bootstrap percolation is the following, often called \emph{$\mathcal{U}$-bootstrap percolation}. For any positive integer $d$, an \emph{update family} on $\mathds{Z}^d$ will be a collection $\mathcal{U}=\{X_1,...,X_m\}$ such that $m$ is a positive integer and for each $i \in \{1,...,m\}$, $X_i$ is a finite nonempty subset of $\mathds{Z}^d \setminus \{0\}$. The $X_i$, $i \in \{1,...,m\}$ are called \emph{update rules}. Let $\mathcal{U}$ be such an update family, it allows us to define the following $\mathcal{U}$-bootstrap percolation dynamics. Each site of $\mathds{Z}^d$ can still be either \emph{healthy} or \emph{infected}, and sites change state in discrete time according to the following rules: for each positive integer $t$, a site that was infected at time $t-1$ remains infected at time $t$, and a site $x$ that was healthy at time $t-1$ is infected at time $t$ if and only if there exists $X \in \mathcal{U}$ so that all sites of $x+X$ were infected at time $t-1$. The bootstrap percolation model with $r$ neighbors is the $\mathcal{U}$-bootstrap percolation dynamics corresponding to $\mathcal{U}=\{$sets of $r$ neighbors of the origin$\}$.

The diversity of the possible update families makes $\mathcal{U}$-bootstrap percolation hard to study. However, Balister, Bollob\'as, Duminil-Copin, Morris, Przykucki, Smith and Uzzell, in their impressive works \cite{Balister_et_al2016,Bollobas_et_al2017,Bollobas_et_al2015}, were able to establish a universality classification of the two-dimensional update families, dividing them into \emph{supercritical}, \emph{critical} and \emph{subcritical} families and characterizing the behavior of $\mathcal{U}$-bootstrap percolation for each class of update families. Later, Balister, Bollobás, Hartarsky, Morris, Smith and Szabó \cite{Balister_et_al2022lower,Balister_et_al2022subcritical,Balister_et_al2022upper,Hartarsky_et_al2022} proved a similar classification in higher dimension. The update family corresponding to the bootstrap percolation with $d$ neighbors, hence to the classical zero-temperature stochastic Ising model, belongs to the critical class.

In light of this understanding of $\mathcal{U}$-bootstrap percolation, Morris \cite{Morris_preprint} introduced generalizations of the voter model and of the zero-temperature stochastic Ising model in the spirit of $\mathcal{U}$-bootstrap percolation, called \emph{$\mathcal{U}$-voter dynamics} and \emph{$\mathcal{U}$-Ising dynamics}. In both models, each site $x \in \mathds{Z}^d$ can be in state $+$ or in state $-$, and attempts to change its state at rate 1. In the $\mathcal{U}$-Ising dynamics, at each attempt, if there exists $X \in \mathcal{U}$ so that all sites in $x+X$ have the state opposite to the state of $x$, then $x$ flips its state to agree with them; otherwise, nothing happens. In the $\mathcal{U}$-voter dynamics, at each attempt, an update rule $X \in \mathcal{U}$ is chosen uniformly at random. If all sites in $x+X$ are in the same state and disagree with the state of $x$, $x$ changes its state to agree with them, otherwise nothing happens. 

The question of fixation in the $\mathcal{U}$-Ising and $\mathcal{U}$-voter dynamics was asked by Morris in \cite{Morris_preprint}, but even less is known about it than about fixation in the classical zero-temperature stochastic Ising model. The only work in this direction is that of Blanquicett \cite{Blanquicett2021}, which showed that for some two-dimensional critical update families, if sites are initially independently at $+$ with a probability close enough to 1, then almost surely all sites fixate at $+$. Obtaining results about fixation for all update families is expected to be even harder than for the classical zero-temperature stochastic Ising model.

In this paper, we tackle a fixation problem which was solved for the zero-temperature stochastic Ising model and study it for the $\mathcal{U}$-Ising and $\mathcal{U}$-voter dynamics. This question is that of fixation when there exists \emph{frozen vertices} which cannot change state. 
\begin{definition}
The $\mathcal{U}$-Ising and $\mathcal{U}$-voter dynamics with frozen vertices are continuous-time dynamics on $\{-,+\}^{\mathds{Z}^d}$, defined as follows. Let $\rho^+,\rho^-\geq0$ be so that $\rho^++\rho^-<1$. The sites of $\mathds{Z}^d$ are independently \emph{frozen at $+$} with probability $\rho^+$, \emph{frozen at $-$} with probability $\rho^-$, and \emph{unfrozen} with probability $1-\rho^+-\rho^-$. The initial states of the unfrozen sites are defined according to a distribution $\mu$, which may depend on the choice of the frozen vertices. Independently from the choice of the frozen vertices and the initial states, independently for each $x\in\mathds{Z}^d$, we consider a Poisson point process with intensity 1, the \emph{clock} at $x$. When the clock at $x$ has a point at some $t$, we say the clock at $x$ \emph{rings at time $t$}. Frozen vertices always remain in the state they are frozen at, and the states of the unfrozen vertices evolve as follows.
 \begin{itemize}
 \item \emph{$\mathcal{U}$-Ising dynamics with frozen vertices:} for each $x \in \mathds{Z}^d$ unfrozen, when the clock at $x$ rings, if there exists $X \in \mathcal{U}$ so that all sites in $x+X$ have the state opposite to the state of $x$, then $x$ flips its state to agree with them; otherwise, nothing happens. 
 \item \emph{$\mathcal{U}$-voter dynamics with frozen vertices:} for each $x \in \mathds{Z}^d$ unfrozen, when the clock at $x$ rings, an update rule $X \in \mathcal{U}$ is chosen uniformly at random, independently from everything else. If all sites in $x+X$ are in the same state and disagree with the state of $x$, $x$ changes its state to agree with them, otherwise nothing happens. 
 \end{itemize} 
 For both dynamics, we denote by $\mathds{P}$ the joint law of the frozen vertices, initial states, clock rings, and choices of update rules (in the $\mathcal{U}$-voter case). In the following, ``almost surely'' will always mean ``almost surely with respect to $\mathds{P}$''. 
\end{definition}
The fact that these dynamics are well defined is not obvious, but is classical, and one can use the arguments in part 4.3 of \cite{Swart2017} to prove it. Given the frozen vertices, both $\mathcal{U}$-Ising and $\mathcal{U}$-voter dynamics with frozen vertices are Markov.

The classical zero-temperature stochastic Ising model with frozen vertices on $\mathds{Z}^d$ was studied by Damron, Eckner, Kogan, Newman and Sidoravicius in \cite{Damron_et_al2014} (see also \cite{Damron_et_al2015}, where a dynamics with a single frozen vertex is considered). They proved that if $\rho^+>0$ and $\rho^-=0$ (no sites frozen at $-$), even if the probability that a site is frozen at $+$ is very small and for any initial states of the unfrozen sites, almost surely all sites fixate at $+$. They also showed that for any $\rho^+>0$, if $\rho^-$ is small enough, the connected components of sites that do not fixate at $+$ (that is flippers and sites that fixate at $-$) are almost surely finite. This is in contrast with the dynamics without frozen vertices, for which \cite{Fontes_et_al2002} implies that all sites fixate at $-$ if the initial probability that a site is at $+$ is small enough.

In this work, we generalize the results of \cite{Damron_et_al2014} and improve on them in both the $\mathcal{U}$-Ising and the $\mathcal{U}$-voter dynamics. In the two-dimensional case, we managed to extend the findings not only to the class of dynamics which can be expected to behave like the classical zero-temperature stochastic Ising model, that is those with critical update families, but also to all those with subcritical update families and some with supercritical update families. We first prove that if $\rho_-=0$, almost surely all sites fixate at $+$ (Theorem \ref{thm_fixation}). Moreover, we show that if $\rho^-$ is small enough, there exists a \emph{deterministic} time so that sites that have not fixated at $+$ at this time form finite connected components (Theorem \ref{thm_connected_components}). This result is new even for the classical zero-temperature stochastic Ising model, and highlights the difference with the dynamics without frozen sites, since Camia, De Santis and Newman \cite{Camia_et_al2002} proved that for the zero-temperature stochastic Ising model without frozen sites, at any time the connected components of $+$ are finite almost surely. Finally, we study the possible existence of flippers. In the zero-temperature stochastic Ising model, if $\rho^->0$ it is not hard to see, as \cite{Damron_et_al2014} did, that there is almost surely an infinite number of flippers. However, this is not the case for all update families. We were able to establish a universality classification sorting the update families into two classes so that if the update family belongs to the first class there is almost surely an infinite number of flippers, but if it belongs to the second class, when $\rho^->0$ is small enough, almost surely there is no flipper (Theorem \ref{thm_flippers}). We also deal with the one-dimensional case, which is simpler but not trivial, and prove similar results (Theorems \ref{thm_fixation_dim1} and \ref{thm_flippers_dim1}).

The proofs partly rely on the important idea already present in \cite{Damron_et_al2014} that if a suitably chosen polygon has sites frozen at $+$ around its corners, once the polygon is filled with $+$ by the dynamics, all the vertices inside have fixated at $+$. However, this only works for non-supercritical update families, and even then the implementation of this idea is notably more difficult and technical that in \cite{Damron_et_al2014} because of the great variety of the update families. Dealing with supercritical update families and proving there are no flippers for a portion of the update families requires entirely novel arguments, outlined in Section \ref{sec_outline}.

This work unfolds as follows. We begin in Section \ref{sec_results} by giving more notation, stating the results and sketching the proofs. Section \ref{sec_good_droplet} contains the construction of the polygons used for the non-supercritical update families. In Section \ref{sec_good_blocks}, we prove that the probability a given region of $\mathds{Z}^2$ is susceptible of easily fixating at $+$ is high. In Section \ref{sec_fixation} we show our fixation results in the two-dimensional case (Theorems \ref{thm_fixation} and \ref{thm_connected_components}). In Section \ref{sec_flippers} we determine which two-dimensional update families exhibit flippers (Theorem \ref{thm_flippers}). Finally, in Section \ref{sec_dim1} we deal with the one-dimensional case (Theorems \ref{thm_fixation_dim1} and \ref{thm_flippers_dim1}). 

\section{Notation and results}\label{sec_results}

Since the one-dimensional case is the simplest one, we begin by spelling out our results for this case, in Subsection \ref{sec_results_dim1}. We then state our two-dimensional results in Subsection \ref{sec_results_dim2}. In Subsection \ref{sec_outline}, we give an outilne of our proofs. Finally, in Subsection \ref{sec_results_notation} we gather some notation that will be used throughout the paper.

\subsection{One-dimensional case}\label{sec_results_dim1}

 This case is much simpler than the two-dimensional one, but it still demands some arguments. One first needs to see the connection between the $\mathcal{U}$-Ising and $\mathcal{U}$-voter dynamics and the $\mathcal{U}$-bootstrap percolation with the same update family. One can see by recursion that each site that can be at state $+$ at some time in the $\mathcal{U}$-Ising or $\mathcal{U}$-voter dynamics can be infected by the $\mathcal{U}$-bootstrap percolation with initially infected sites the sites initially at state $+$. Indeed, if a site $x$ switches its state to $+$, there is an update rule $X$ so that the sites of $x+X$ were at $+$ before the switch. Therefore, for fixation at $+$ to be possible for all sites of $\mathds{Z}$, this $\mathcal{U}$-bootstrap percolation process must be able to infect all sites in $\mathds{Z}$, and this depends on the properties of the update family. One-dimensional update families are classified as follows.
 
 \begin{definition}
 A one-dimensional update family $\mathcal{U}$ is called  
 \begin{itemize}
  \item \emph{supercritical} if there exists $X \in \mathcal{U}$ so that $X \subset \{1,2,...\}$ or $X \subset \{...,-2,-1\}$;
  \item \emph{subcritical} otherwise.
 \end{itemize}
 \end{definition}
 
 If the update family is subcritical, the $\mathcal{U}$-bootstrap percolation process is unable to infect all sites of $\mathds{Z}$, since when we have a large interval of initially healthy sites, no site of this interval can be infected. Consequently, for subcritical update families, the $\mathcal{U}$-Ising and the $\mathcal{U}$-voter dynamics will be unable to set at $+$ all sites of $\mathds{Z}$. We thus study fixation at $+$ only for supercritical update families. If there are no sites frozen at $-$, we were able to prove that fixation at $+$ occurs for all sites of $\mathds{Z}$, which is the following result.
 
 \begin{theorem}\label{thm_fixation_dim1}
If $\mathcal{U}$ is a supercritical one-dimensional update family, in the $\mathcal{U}$-Ising and the $\mathcal{U}$-voter dynamics with frozen vertices, if $0 < \rho^+ <1$ and $\rho^-=0$, for any choice of initial distribution $\mu$, almost surely all sites fixate at $+$.
\end{theorem}

If there are sites frozen at $-$, one cannot expect fixation at $+$ for all sites of $\mathds{Z}$. However, it is not \emph{a priori} obvious whether there will be flippers that change state an infinite number of times, or if all sites end up fixating either at $+$ or at $-$. This actually depends on the update family, and we were able to show the following universality classification which describes the possible behaviors and is valid for both subcritical and supercritical update families.

\begin{theorem}\label{thm_flippers_dim1}
If $\mathcal{U}$ is a one-dimensional update family, in the $\mathcal{U}$-Ising and the $\mathcal{U}$-voter dynamics with frozen vertices, if $0 < \rho^+ <1$ and $0 < \rho^- < 1-\rho^+$, for any choice of initial distribution $\mu$, there are two possible cases:
\begin{itemize}
\item if $\mathcal{U}$ contains two disjoint update rules, almost surely there is an infinite number of flippers;
\item if $\mathcal{U}$ contains no disjoint update rules, almost surely there is no flipper.
\end{itemize}
\end{theorem}

\subsection{Two-dimensional case}\label{sec_results_dim2}

This case is much more complex than the one-dimensional case, in part because update families are more diverse. In order to state our results, we need to explain the classification of the two-dimensional update families introduced in \cite{Bollobas_et_al2015} by Bollobás, Smith and Uzzell. Let $\mathcal{U}$ be an update family on $\mathds{Z}^2$. We denote $\langle \cdot,\cdot\rangle$ the scalar product on $\mathds{R}^2$. For any direction $u \in S^1$, for any $a \in \mathds{R}$, we define $\mathds{H}_u=\{x \in \mathds{R}^2\,|\,\langle x,u \rangle <0 \}$ as the open half-plane opposed to $u$. We say $u$ is a \emph{stable direction} for $\mathcal{U}$ when there exists no $X\in \mathcal{U}$ so that $X \subset \mathds{H}_u$; otherwise $u$ is called \emph{unstable}. We then have the following.

\begin{definition}
The update family $\mathcal{U}$ on $\mathds{Z}^2$ is called 
\begin{itemize}
\item \emph{supercritical} if there exists an open semicircle of unstable directions;
\item \emph{critical} if it is not supercritical, but there exists a semicircle containing a finite number of stable directions;
\item \emph{subcritical} otherwise.
\end{itemize}
\end{definition}

As in the one-dimensional case, a site can be set at $+$ by the $\mathcal{U}$-Ising or the $\mathcal{U}$-voter dynamics only if it can be infected in the $\mathcal{U}$-bootstrap percolation with initially infected sites the sites initially at $+$. Moreover, it was proven in \cite{Balister_et_al2016} by Balister, Bollobás, Przykucki and Smith that if $\mathcal{U}$ is subcritical, there exists $q_c(\mathcal{U})>0$ so that if sites are initially independently infected with probability $q$, for $q < q_c(\mathcal{U})$ the $\mathcal{U}$-bootstrap percolation is unable to infect all sites of $\mathds{Z}^2$, while if $q>q_c(\mathcal{U})$ it is able to\footnote{In \cite{Balister_et_al2016}, $q_c(\mathcal{U})$ was defined as $\inf\{q\in[0,1]\,|\,p(q)\geq 1/2\}$, where $p(q)$ is the probability $\mathcal{U}$-bootstrap percolation can infect all sites of $\mathds{Z}^2$ if sites are initially independently infected with probability $q$. However, the event that all sites of $\mathds{Z}^2$ can be infected is invariant by translation, so standard ergodicity arguments (along the lines of Proposition 7.3 of \cite{Lyons_Peres} for example) imply $p(q)$ is always 0 or 1, and $p(q)$ can only increase with $q$ since bootstrap percolation is monotone (see Section \ref{sec_results_notation}), therefore we have $p(q)=0$ for $q < q_c(\mathcal{U})$ and $p(q)=1$ for $q > q_c(\mathcal{U})$.}. Therefore we do not study subcritical models if the probability $\rho^+$ that a site is frozen at $+$ is smaller than $q_c(\mathcal{U})$. Actually, technical reasons prevent us to prove our results for all $\rho^+>q_c(\mathcal{U})$, so, like many results for subcritical update families, ours will hold for $\rho^+>\tilde q_c(\mathcal{U})$, where $\tilde q_c(\mathcal{U}) \geq q_c(\mathcal{U})$ was introduced by Hartarsky in \cite{Hartarsky2018} and is conjectured to be equal to $q_c(\mathcal{U})$. The definition of $\tilde q_c(\mathcal{U}) $ is as follows (we give it for completeness, but it is not necessary to understand our results). For any $q\in[0,1]$, $n\in\mathds{N}$, let $\theta_n(q)$ be the probability for the origin to remain always healthy in the $\mathcal{U}$-bootstrap percolation dynamics when all sites of $\{-n,...,n\}^2$ are initially independently infected with probability $q$ and no other site is initially infected. Then 
\[
\tilde q_c(\mathcal{U})=\inf\left\{q \in [0,1] \,\left|\, \sum_{n\in\mathds{N}}n\theta_n(q)<+\infty\right.\right\}.
\]
For this reason, we will work when the following condition holds:
\[
(\mathcal{C}) : \qquad\qquad\begin{array}{rl}
& \mathcal{U}\text{ is subcritical and }\tilde q_c(\mathcal{U})<\rho^+<1, \\
 \text{or} & \mathcal{U}\text{ is critical and }0 <\rho^+ <1, \\
  \text{or} & \mathcal{U}\text{ is supercritical, contains no disjoint update rules and }0 < \rho^+ <1.
  \end{array}
\]
When there are no sites frozen at $-$, we have the following result of fixation at $+$ for all sites of $\mathds{Z}^2$. 

\begin{theorem}\label{thm_fixation}
Let $\mathcal{U}$ be a two-dimensional update family. In the $\mathcal{U}$-Ising and the $\mathcal{U}$-voter dynamics with frozen vertices, if condition $(\mathcal{C})$ holds, then if $\rho^-=0$, for any choice of initial distribution $\mu$, almost surely all sites fixate at~$+$.
\end{theorem}

\begin{remark}
There are a number of interesting supercritical update families with no disjoint update rules, for example $\mathcal{U}=\{\{(-1,0),(-1,1)\},\{(-1,0),(-1,-1)\}\}$ (see Figure \ref{fig_family}). For the other class of supercritical update families, those containing two disjoint update rules, it is not clear whether fixation occurs. Indeed, as proven in Section 5.3 of \cite{Bollobas_et_al2015}, one can find a finite set of sites which, if initially infected, allows the $\mathcal{U}$-bootstrap percolation dynamics to propagate the infection to an infinite number of sites. Therefore, even if there is only a finite set of sites at $-$ in the $\mathcal{U}$-Ising or the $\mathcal{U}$-voter dynamics, this dynamics can potentially propagate the $-$ to an unlimited number of sites, by switching to $-$ the sites that get infected in the order they get infected (see Remark \ref{rem_supercritical} for a more detailed construction). 
\end{remark}

\begin{figure}
 \begin{center}
 \parbox{0.3\textwidth}{\center
  \begin{tikzpicture}[scale=0.6]
   \draw (-1.5,1)--(1.5,1);
   \draw (-1.5,0)--(1.5,0);
   \draw (-1.5,-1)--(1.5,-1);
   \draw (1,-1.5)--(1,1.5);
   \draw (0,-1.5)--(0,1.5);
   \draw (-1,-1.5)--(-1,1.5);
   \draw (-1,0) node {$\bullet$};
   \draw (-1,1) node {$\bullet$};
   \draw (0,0) node {$\times$} node [above right] {0};
  \end{tikzpicture}
  
  $\{(-1,0),(-1,1)\}$}
  \parbox{0.3\textwidth}{\center
  \begin{tikzpicture}[scale=0.6]
   \draw (-1.5,1)--(1.5,1);
   \draw (-1.5,0)--(1.5,0);
   \draw (-1.5,-1)--(1.5,-1);
   \draw (1,-1.5)--(1,1.5);
   \draw (0,-1.5)--(0,1.5);
   \draw (-1,-1.5)--(-1,1.5);
   \draw (-1,0) node {$\bullet$};
   \draw (-1,-1) node {$\bullet$};
   \draw (0,0) node {$\times$} node [above right] {0};
  \end{tikzpicture}
  
  $\{(-1,0),(-1,-1)\}$}
 \end{center}
 \caption{A two-dimensional supercritical update family with no disjoint update rules: $\mathcal{U}=\{\{(-1,0),(-1,1)\},\{(-1,0),(-1,-1)\}\}$. The bullets represent the sites contained in the update rules.}
 \label{fig_family}
\end{figure}
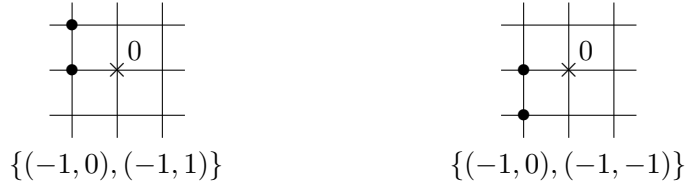

Except when mentioned otherwise, connectedness will be connectedness for the usual graph structure of $\mathds{Z}^2$. When the probability $\rho^-$ that a site is frozen at $-$ is small, we were able to prove the existence of a deterministic time after which the connected components of sites that have not yet fixated at $+$ are finite.

\begin{theorem}\label{thm_connected_components}
Let $\mathcal{U}$ be a two-dimensional update family. In the $\mathcal{U}$-Ising and the $\mathcal{U}$-voter dynamics with frozen vertices, if condition $(\mathcal{C})$ holds, then there exists $\rho_0^-=\rho_0^-(\mathcal{U},\rho^+)>0$ so that for any $0 \leq \rho^- \leq \rho_0^-$, for any choice of initial distribution $\mu$, there exists a deterministic time $T=T(\mathcal{U},\rho^+)<+\infty$ such that almost surely for any $t \geq T$, the connected components of sites that have not fixated at $+$ at time $t$ are finite.
\end{theorem}

\begin{remark}
 This cannot hold in the one-dimensional setting, since a single $-$ is enough to break the connexity, and for any time there is a positive probability that a site initially at $-$ has not yet received a clock ring allowing it to switch to $+$ before this time.
\end{remark}

Theorem \ref{thm_connected_components} means that ``most'' sites fixate at $+$, but does not imply all sites fixate at $+$. If $\rho^->0$, we can expect some sites not to fixate at $+$, but it is not obvious whether there will be flippers. We were able to answer this question and to establish a universality classification of the update families describing the possible behaviors, which is the following.

\begin{theorem}\label{thm_flippers}
Let $\mathcal{U}$ be a two-dimensional update family. In the $\mathcal{U}$-Ising and the $\mathcal{U}$-voter dynamics with frozen vertices, if condition $(\mathcal{C})$ holds, then for $0 < \rho^- \leq \rho_0^-$, for any choice of initial distribution $\mu$, there are two possible cases:
\begin{itemize}
\item if $\mathcal{U}$ contains two disjoint update rules, almost surely there is an infinite number of flippers;
\item if $\mathcal{U}$ contains no disjoint update rules, almost surely there is no flipper.
\end{itemize}
\end{theorem}

\begin{remark}
 The $\mathcal{U}$-Ising and $\mathcal{U}$-voter dynamics are symmetric with respect to $+$ and $-$, therefore similar theorems hold with the roles of $+$ and $-$ reversed.
 \end{remark}

\begin{remark}
The results of \cite{Damron_et_al2014} for fixation in the classical zero-temperature stochastic Ising model with frozen vertices are also valid in dimension $d \geq 3$, and a universality classification was proved for $\mathcal{U}$-bootstrap percolation in dimension $d \geq 3$, by Balister, Bollobás, Hartarsky, Morris, Smith and Szabó \cite{Balister_et_al2022lower,Balister_et_al2022subcritical,Balister_et_al2022upper,Hartarsky_et_al2022}. As in the two-dimensional case, update families are sorted according to the structure of the set of their stable directions, which are defined similarly, but with half-spaces instead of half-planes. Despite this, our arguments cannot be extended to higher dimension, even for update families in the same class as that of the classical zero-temperature stochastic Ising model. Let us explain why in the three-dimensional case. In this case, the update family corresponding to the zero-temperature stochastic Ising model is $\mathcal{U}=\{$sets of 3 neighbors of the origin$\}$. For this update family, a cube filled with $+$ whose corners are frozen at $+$ has fixated at $+$, and this fact is the core of the argument of \cite{Damron_et_al2014}. However, if one adds to this update family the update rule $\{(1,-1,0),(-1,1,0)\}$, then the set of stable directions remains the same, but one can see that the sites on the vertical edge of the cube with maximal abscissa and ordinate may switch to $-$ even if the cube is filled with $+$. Therefore even complete knowledge of the set of stable directions is not enough. This suggests that the question of fixation in higher dimension may be much more complex than in the two-dimensional case. 
\end{remark}

\subsection{Outline of the proofs}\label{sec_outline}

We first explain the arguments used for the two-dimensional case. Their cornerstone is finding regions of $\mathds{Z}^2$ so that once the dynamics fills the region by chance, all sites in the region have fixated at $+$. For the classical zero-temperature stochastic Ising model, this was accomplished in \cite{Damron_et_al2014} by noticing that if a square has its four corners frozen at $+$, once the square is filled with $+$, all its sites have fixated at $+$. We extend this idea not only to critical update families (the class corresponding to the classical zero-temperature stochastic Ising model), but also to subcritical update families, by constructing a polygon such that if its ``enlarged corners'' are made of sites frozen at $+$, once the polygon is filled with $+$, all its sites have fixated at $+$. This is done in Section \ref{sec_good_droplet}. For supercritical update families, it is impossible to construct such a polygon (see Remark \ref{rem_supercritical}), hence we need a new mechanism. For this, we notice that if there are no disjoint update rules, when a site $x$ is set at $+$, there exists $X \in \mathcal{U}$ so that $x+X$ is made of sites at $+$. As long as $x+X$ is at $+$, for any $X'\in \mathcal{U}$, since $X$ and $X'$ are not disjoint, $x+X'$ contains at least a site of $x+X$, so $x+X'$ is not entirely at $-$, thus $x$ cannot be set at $-$ before one of the $+$ of $x+X$ disappears. Therefore a site $x$ that is set at $+$ thanks to sites $x+X$ frozen or fixated at $+$ will fixate at $+$, and fixation at $+$ will propagate from the frozen sites. Thanks to these arguments, if the configuration of frozen sites is favorable (enough sites frozen at $+$), we can construct regions whose sites fixate at $+$. 

This will allow us to prove our fixation results, which is done in Section \ref{sec_fixation}. Indeed, we can show that the probability that the configuration of frozen sites is favorable is so high (Proposition \ref{prop_good_blocks}) that most of $\mathds{Z}^2$ will be favorable. Most of $\mathds{Z}^2$ thus fixates rather quickly, which allows to find a deterministic time $T$ such that the connected components of sites that have not fixated at $+$ at time $T$ are finite, which is Theorem \ref{thm_connected_components}. We even obtain a slightly stronger result (Theorem \ref{thm_connected_components_blocks}), which yields that after time $T$, the dynamics only evolves on finite connected components isolated from each other by sites that have fixated at $+$. If $\rho^-=0$, once by chance the dynamics has filled one of these components with $+$, none of its sites can ever filp to $-$ anymore, which allows to prove almost sure fixation at $+$ for all sites (Theorem \ref{thm_fixation}). 

Proving Theorem \ref{thm_flippers} about the existence of flippers requires an additional novel argument, given in Section \ref{sec_flippers}. If there are two disjoint update rules $X$ and $X'$ and if $\rho^->0$, it is easy to see there will be an infinite number of flippers, since a site $x$ with $x+X$ frozen at $+$ and $x+X'$ frozen at $-$ will be a flipper. However, proving that there is no flipper if there are no disjoint update rules is more complex. To do this, we consider the dynamics in each of the aforementioned finite connected components. Then we craft events such that if one of these events occurs, all sites in the component have fixated at $+$, and prove that almost surely one of these events will occur sooner or later. 

In the one-dimensional case, the arguments are similar, but simpler. Indeed, one does not need to construct regions which, once filled with $+$ by the dynamics, have fixated at $+$. Instead, we notice that there is an infinite number of large intervals of sites frozen at $+$, and that their complement is composed of finite connected components that cannot interact with each other. Once we make this observation, the arguments that prove Theorems \ref{thm_fixation} and \ref{thm_flippers} allow to prove Theorems \ref{thm_fixation_dim1} and \ref{thm_flippers_dim1}, which is done in Section \ref{sec_dim1}. 

\subsection{Notation}\label{sec_results_notation}

We gather here some notation which will be used throughout the paper. $\mathds{N}^*$ is the set of positive integers. For any set $A$, we denote by $|A|$ the cardinal of $A$. We recall that $\langle \cdot,\cdot \rangle$ is the scalar product on $\mathds{R}^2$. $\|.\|_2$ will denote the Euclidean norm on $\mathds{R}^2$, and we will use the distance on $\mathds{R}^2$ associated to it. We also denote by $\|.\|_\infty$ the sup norm. For any direction $u \in S^1$, for any $a \in \mathds{R}$, we denote $\mathds{H}_u(a)=au+\mathds{H}_u=\{x \in \mathds{R}^2\,|\,\langle x,u \rangle <a \}$ the translation of the open half-plane $\mathds{H}_u$ by $a$, and $\bar{\mathds{H}}_u(a)=\{x \in \mathds{R}^2\,|\,\langle x,u \rangle \leq a \}$ the corresponding closed half-plane.

Except in Section \ref{sec_dim1}, where the families will be one-dimensional, $\mathcal{U}$ will be an update family on $\mathds{Z}^2$, and we assume condition $(\mathcal{C})$ holds. We define a constant $r=r(\mathcal{U})=\max\{\|x\|_2 \,|\, x \in X, X \in \mathcal{U}\}$ which represents the range of $\mathcal{U}$. We also define $\mathcal{U}$-bootstrap percolation on a domain $D \subset \mathds{Z}^2$, which will always have \emph{healthy boundary conditions}: the dynamics is the same as the $\mathcal{U}$-bootstrap percolation dynamics explained in the Introduction, except that sites in $\mathds{Z}^2 \setminus D$ always remain healthy. From now on, the $\mathcal{U}$-bootstrap percolation with respect to the update family $\mathcal{U}$ will simply be called ``bootstrap percolation''. We notice that the bootstrap percolation dynamics is \emph{monotone}: if one has two sets of sites $A_1 \subset A_2$, and considers a first dynamics where the initially infected sites are the sites in $A_1$, and a second dynamics in which the initially infected sites are the sites in $A_2$, then it can be checked by recursion that for any integer $t$, the sites infected at time $t$ in the first dynamics are also infected at time $t$ in the second dynamics. 

In the remainder of this paper, we will work with the $\mathcal{U}$-voter dynamics, since the arguments for the $\mathcal{U}$-Ising dynamics are similar and simpler. If at some time a given site is at $+$ (respectively at $-$) and it is impossible for it to switch to $-$ (respectively to $+$) again without having two clocks ring at the same time (which has probability 0), we say the site is \emph{well fixed at $+$} (respectively at $-$) at this time. Almost surely, once a site is well fixed, it has fixated. For all $t \geq 0$, we denote by $\mathcal{F}_t$ the $\sigma$-algebra generated by the choice of the frozen sites, initial states of the sites, and by the clock rings and choices of update rules occurring in the time interval $[0,t]$.  

\section{Non-supercritical update families: good droplets} \label{sec_good_droplet}

In this section, we assume the update family $\mathcal{U}$ is not supercritical. The section is devoted to the construction of a polygon so that if its ``enlarged corners'' are frozen at $+$, once all the sites inside the polygon are at $+$, they have fixated at $+$. We will call such a polygon ``good droplet''. We now give the necessary notation. Since $\mathcal{U}$ is not supercritical, there is no open semicircle of unstable directions, which implies that there exists $N=3$ or 4 stable directions $u_1,...,u_N$ such that the origin of $\mathds{R}^2$ is contained in the interior of the convex envelope of $u_1,...,u_N$. We assume that $N$ is minimal, and that $u_1,...,u_N$ are numbered counterclockwise. A \emph{droplet} will be a polygon with sides orthogonal to $u_1,...,u_N$: for any $a>0$, we set $D(a)=\bigcap_{i=1}^N \bar{\mathds{H}}_{-u_i}(a)$ (see Figure \ref{fig_droplets}). We stress that these droplets are different from those that are generally used in the study of bootstrap percolation, which are of the form $D'(a)=\bigcap_{i=1}^N \mathds{H}_{u_i}(a)$ (see also Figure \ref{fig_droplets}). $D'(a)$ satisfies that if the only initial infected sites are inside $D'(a)$, there will never be infected sites outside $D'(a)$, as the first such site $x$ that would be infected would be outside an $\mathds{H}_{u_i}(a)$ while the update rule $x+X$ that infects it would be inside $\mathds{H}_{u_i}(a)$, which is impossible since $u_i$ is a stable direction. Thus the sites inside $D'(a)$ cannot influence the sites outside. To define $D(a)$, we replaced the $u_i$ by $-u_i$ because we want the opposite: prevent the sites outside the droplet to influence sites inside the droplet. However, sites ``at the corners'' of $D(a)$ will be surrounded mostly by sites in $D(a)^c$, so they can be influenced by sites in $D(a)^c$. This is why we will require the sites ``at the corners'' to be frozen.

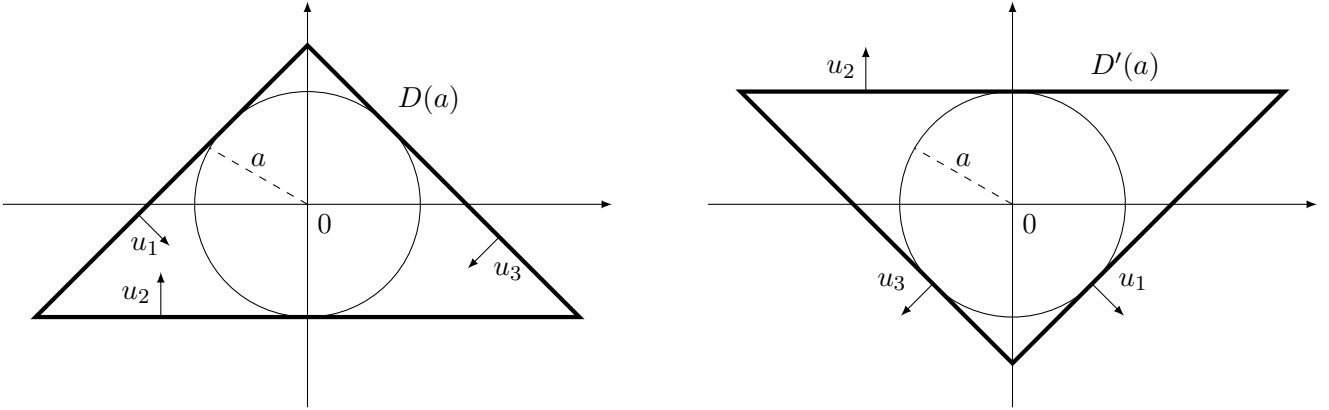
\begin{figure}
\begin{tikzpicture}[scale=1.5]
 \draw[->,>=latex] (-2.7,0)--(2.7,0) ;
 \draw[->,>=latex] (0,-1.8)--(0,1.8) ;
 \draw (0,0) node [below right] {0} ;
 \draw[ultra thick] (2.41,-1)--(-2.41,-1)--(0,1.41)--cycle ;
 \draw (0.71,0.71) node [above right] {$D(a)$} ;
 \draw (0,0) circle (1) ;
 \draw[dashed] (0,0)--(-0.87,0.5) node[midway,above] {$a$};
 \draw [->,>=latex] (-1.5,-0.09)--(-1.22,-0.37) node[left] {$u_1$} ;
 \draw [->,>=latex] (-1.3,-1)--(-1.3,-0.6) node[midway,left] {$u_2$} ;
 \draw [->,>=latex] (1.7,-0.29)--(1.42,-0.57) node[midway, below right] {$u_3$} ;
\end{tikzpicture}
\hspace{1cm}
\begin{tikzpicture}[scale=1.5]
 \draw[->,>=latex] (-2.7,0)--(2.7,0) ;
 \draw[->,>=latex] (0,-1.8)--(0,1.8) ;
 \draw (0,0) node [below right] {0} ;
 \draw[ultra thick] (2.41,1)--(-2.41,1)--(0,-1.41)--cycle ;
 \draw (1,1) node [above] {$D'(a)$} ;
 \draw (0,0) circle (1) ;
 \draw[dashed] (0,0)--(-0.87,0.5) node[midway,above] {$a$};
 \draw [->,>=latex] (0.71,-0.71)--(0.99,-0.99) node[midway, above right] {$u_1$} ;
 \draw [->,>=latex] (-1.3,1)--(-1.3,1.4) node[midway,left] {$u_2$} ;
 \draw [->,>=latex] (-0.71,-0.71)--(-0.99,-0.99) node[midway, above left] {$u_3$} ;
\end{tikzpicture}

\caption{Droplets $D(a)$ and $D'(a)$ in the case $N=3$.}\label{fig_droplets}
\end{figure}

In order to specify which sites need to be frozen exactly, we need more notation. We identify $\{1,...,N\}$ with $\mathds{Z}/N\mathds{Z}$, so $u_{N+1}=u_1$. For any $i \in \{1,...,N\}$, we set $C_{i,i+1}(a) = (\bar{\mathds{H}}_{-u_i}(a)\setminus\mathds{H}_{-u_i}(a-r)) \cap (\bar{\mathds{H}}_{-u_{i+1}}(a)\setminus\mathds{H}_{-u_{i+1}}(a-r))$ (see Figure \ref{fig_corners}) an ``enlarged corner of $D(a)$ between the side of the polygon orthogonal to $u_i$ and the side orthogonal to $u_{i+1}$''. For any $x\in\mathds{Z}^2$, we say $x+D(a)$ is good when for all $i \in \{1,...,N\}$, all the sites of $x+C_{i,i+1}(a)$ are frozen at $+$. The following result is one of the most important ingredients for our proofs.

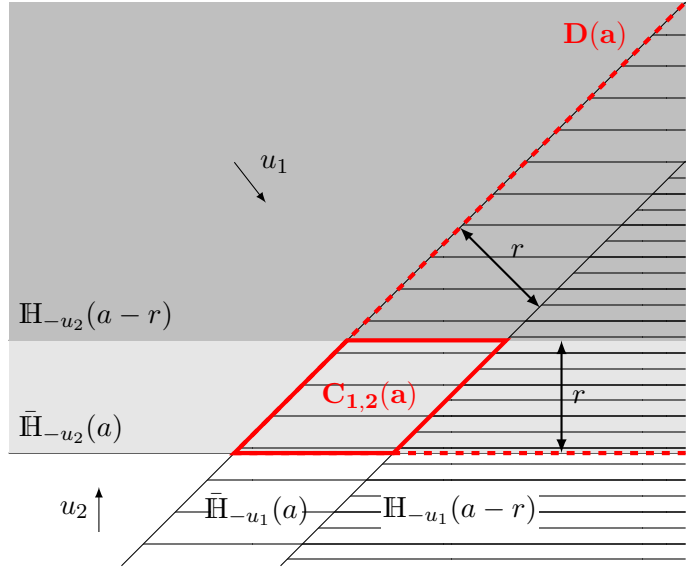
\begin{figure}
\begin{tikzpicture}[scale=1.5]
 \draw (-2,0)--(4,0) ;
 \fill[color=gray!20] (-2,0)--(4,0)--(4,4)--(-2,4)--cycle ;
 \draw (-2,0) [above right] node {$\bar{\mathds{H}}_{-u_2}(a)$};
 \draw (-2,1)--(4,1) ;
 \fill[color=gray!50] (-2,1)--(4,1)--(4,4)--(-2,4)--cycle ;
  \draw (-2,1) [above right] node {$\mathds{H}_{-u_2}(a-r)$};
 \draw (-1,-1)--(4,4) ;
 \fill[pattern={Lines[angle=90,distance={12pt}]}] (-1,-1)--(4,4)--(4,4)--(4,-1)--cycle ;
 \fill[white] (-0.2,-0.7)--(0.6,-0.7)--(0.6,-0.3)--(-0.2,-0.3) ;
 \draw (0.2,-0.5) node {$\bar{\mathds{H}}_{-u_1}(a)$};
 \draw (0.41,-1)--(4,2.59) ;
 \fill[pattern={Lines[angle=90,distance={6pt}]}] (0.41,-1)--(4,2.59)--(4,-1)--cycle ;
 \fill[white] (1.3,-0.7)--(2.7,-0.7)--(2.7,-0.3)--(1.3,-0.3) ;
 \draw (2,-0.5) node {$\mathds{H}_{-u_1}(a-r)$};
 \draw [ultra thick,red,dashed] (4,0)--(0,0)--(4,4) ;
 \draw (3.2,3.7) node{$\color{red}\boldsymbol{D(a)}$} ;
 \draw [ultra thick,red] (0,0)--(1.41,0)--(2.41,1)--(1,1)--cycle ; 
 \draw (1.2,0.5) node{$\color{red}\boldsymbol{C_{1,2}(a)}$} ;
 \draw[->,>=latex] (-1.2,-0.7)--(-1.2,-0.3) node [midway,left] {$u_2$};
 \draw[->,>=latex] (0,2.58)--(0.28,2.22) node [midway,above right] {$u_1$};
 \draw[thick,<->,>=latex] (2,2)--(2.71,1.29) node[midway, above right] {$r$} ;
 \draw[thick,<->,>=latex] (2.9,0)--(2.9,1) node[midway, right] {$r$} ;
\end{tikzpicture}
\caption{A representation of $C_{1,2}(a) = (\bar{\mathds{H}}_{-u_1}(a)\setminus\mathds{H}_{-u_1}(a-r)) \cap (\bar{\mathds{H}}_{-u_{2}}(a)\setminus\mathds{H}_{-u_{2}}(a-r))$, here the ``bottom left corner of $D(a)$''. $C_{1,2}(a)$ is the shape with the solid thick outline, and $D(a)$ the shape with the dashed thick outline. $\bar{\mathds{H}}_{-u_1}(a)$ and $\mathds{H}_{-u_1}(a-r)$ are represented with hatchings: $\mathds{H}_{-u_1}(a-r)$ is the region with dense hatchings, and $\bar{\mathds{H}}_{-u_1}(a)$ contains both this region and the region with the spaced hatchings. Similarly, $\mathds{H}_{-u_2}(a-r)$ is the region in dark gray and $\bar{\mathds{H}}_{-u_2}(a)$ contains both this region and the region in lighter gray.}
\label{fig_corners}
\end{figure}

\begin{proposition}\label{prop_key}
There exists $a_0=a_0(\mathcal{U},u_1,..,u_N)<+\infty$ so that for any $a \geq a_0$, if $D(a)$ is good, no site of $D(a)$ is frozen at $-$, and at some time $t \geq 0$ all sites in $D(a)$ are in state $+$, then all sites in $D(a)$ are well fixed at $+$ at time~$t$.
\end{proposition} 

\begin{proof}
We prove the statement by contradiction. If some site of $D(a)$ is not well fixed at $+$ at time $t$, there exists $x \in D(a)$ such that $x$ can be the first site in $D(a)$ to switch its state to $-$ after time $t$. Then by the definition of the dynamics, there exists $X \in \mathcal{U}$ so that all sites in $x+X$ are at $-$ just before the switch. We will prove this is geometrically impossible. Since $x$ is the first site in $D(a)$ to change state, $x+X \subset D(a)^c$. Moreover, all sites in $x+X$ are at distance at most $r$ from $x$, therefore $x$ belongs to some $\bar{\mathds{H}}_{-u_i}(a)\setminus\mathds{H}_{-u_i}(a-r)$. Furthermore, since $C_{i-1,i}(a)$ and $C_{i,i+1}(a)$ are frozen, $x$ does not belong to them, hence $x \in \mathds{H}_{-u_{i-1}}(a-r)$ and $x \in \mathds{H}_{-u_{i+1}}(a-r)$. If $N=3$, we then have  $x \in \mathds{H}_{-u_{j}}(a-r)$ for all $j \neq i$. 

We now show it in the case $N=4$, where we have to deal with $j=i+2$. Since $D(1)$ has four sides of positive length, there exists some $\varepsilon=\varepsilon(u_1,u_2,u_3,u_4) >0$ so that $\bar{\mathds{H}}_{-u_1}(1-\varepsilon) \cap \bar{\mathds{H}}_{-u_3}(1-\varepsilon) \cap \bar{\mathds{H}}_{-u_2}(1) \cap \bar{\mathds{H}}_{-u_4}(1)$ has still four sides of positive length. Then $D(1)\cap(\bar{\mathds{H}}_{-u_1}(1)\setminus\mathds{H}_{-u_1}(1-\varepsilon))$ and  $D(1)\cap(\bar{\mathds{H}}_{-u_3}(1)\setminus\mathds{H}_{-u_3}(1-\varepsilon))$ are disjoint, thus $D(a)\cap(\bar{\mathds{H}}_{-u_1}(a)\setminus\mathds{H}_{-u_1}(a-a\varepsilon))$ and  $D(a)\cap(\bar{\mathds{H}}_{-u_3}(a)\setminus\mathds{H}_{-u_3}(a-a\varepsilon))$ are disjoint. If $a \geq r/\varepsilon$, this implies $D(a)\cap(\bar{\mathds{H}}_{-u_1}(a)\setminus\mathds{H}_{-u_1}(a-r))$ and  $D(a)\cap(\bar{\mathds{H}}_{-u_3}(a)\setminus\mathds{H}_{-u_3}(a-r))$ are disjoint. Similarly, one can find some $\varepsilon'=\varepsilon'(u_1,u_2,u_3,u_4) >0$ so that if $a \geq r/\varepsilon'$, $D(a)\cap(\bar{\mathds{H}}_{-u_2}(a)\setminus\mathds{H}_{-u_2}(a-r))$ and  $D(a)\cap(\bar{\mathds{H}}_{-u_4}(a)\setminus\mathds{H}_{-u_4}(a-r))$ are disjoint. We set $a_0=\max(r/\varepsilon,r/\varepsilon')$, then if $a \geq a_0$, since $x \in D(a) \cap (\bar{\mathds{H}}_{-u_i}(a)\setminus\mathds{H}_{-u_i}(a-r))$, we have $x \not\in D(a) \cap (\bar{\mathds{H}}_{-u_{i+2}}(a)\setminus\mathds{H}_{-u_{i+2}}(a-r))$, thus $x \in \mathds{H}_{-u_{i+2}}(a-r)$. Consequently, $x \in \mathds{H}_{-u_{j}}(a-r)$ for all $j \neq i$.

Since $x \in \mathds{H}_{-u_{j}}(a-r)$ for all $j \neq i$, we have $x+X \subset \mathds{H}_{-u_{j}}(a)$ for all $j \neq i$. But $x+X \subset D(a)^c$, hence $x+X \subset \bar{\mathds{H}}_{-u_{i}}(a)^c=\mathds{H}_{u_i}(-a)$. However, $x \in \bar{\mathds{H}}_{-u_i}(a)$, so $\langle x,u_i\rangle \geq -a$, hence $x+X \subset x+\mathds{H}_{u_i}$, which contradicts the fact that $u_i$ is a stable direction, and ends the proof.
\end{proof}

\begin{remark}\label{rem_supercritical}
 If $\mathcal{U}$ is supercritical, a similar construction is impossible. Indeed, it was proven in Section 5.3 of \cite{Bollobas_et_al2015} that for supercritical update families, one can find a finite set of sites such that the bootstrap percolation dynamics starting with only these sites initially infected will propagate the infection at an unlimited distance along a given direction. In the $\mathcal{U}$-voter dynamics, if such a set of sites is at $-$ and is positioned so that this direction points to a side of $D(a)$, even if these sites are far away from $D(a)$, the $-$ may propagate from these sites and enter $D(a)$. Indeed, we can denote $x_1,x_2,...$ the sites that are infected by this bootstrap percolation dynamics in the order they are infected, excluding those that are already at $-$, and $X_1,X_2,...$ the respective update rules allowing to infect them. If there are successive clock rings at $x_1,x_2,...$ and the update rules chosen are $X_1,X_2,...$, these sites will switch to $-$. 
\end{remark}

 We also prove a technical lemma on the $C_{i,i+1}(a)$ that will be needed later.

 \begin{lemma}\label{lem_technical}
 There exists $\tilde a_0=\tilde a_0(\mathcal{U},u_1,..,u_N)<+\infty$ so that for any $a \geq \tilde a_0$, for any $i\in\{1,...,N\}$, $C_{i,i+1}(a) \subset D(a)$. 
\end{lemma}

\begin{proof}
 If $N=3$, $D(a)$ has 3 sides, and if $N=4$ then $D(a)$ has four sides, since if $D(a)$ had less than four sides, the directions orthogonal to these sides would contain 0 in the interior of their convex envelope, which would contradict the minimality of $N$. We notice that if we set $i\in\{1,...,N\}$, then for all $a>0$ the $C_{i,i+1}(a)$ are translations of each other. We denote $d_i = \max\{\|y-y'\|_2\,|\,y,y'\in C_{i,i+1}(1)\}$ the diameter of the $C_{i,i+1}(a)$. For each $i\in\{1,...,N\}$, we denote by $c_{i,i+1}$ the site that is the ``corner of $D(1)$ between the side of the polygon orthogonal to $u_i$ and the side orthogonal to $u_{i+1}$'', defined by $\langle c_{i,i+1},-u_i \rangle=1$ and $\langle c_{i,i+1},-u_{i+1} \rangle=1$. Since $D(1)$ has $N$ sides, for each $j \neq i,i+1$, $\langle c_{i,i+1},-u_j \rangle < 1$. If we choose $\tilde a_0 \geq \max\{d_i,\frac{d_i}{1-\langle c_{i,i+1},-u_j \rangle} \,|\,i,j \in\{1,...,N\},j \neq i,i+1\}$, then for $a \geq \tilde a_0$, $i\in\{1,...,N\}$ and $j \neq i,i+1$, the points in $C_{i,i+1}(a)$ are at distance at most $d_i$ of $ac_{i,i+1}$, hence are in $\bar{\mathds{H}}_{-u_j}(a\langle c_{i,i+1},-u_j \rangle+d_i) \subset \bar{\mathds{H}}_{-u_j}(a)$. We deduce $C_{i,i+1}(a) \subset D(a)$. 
\end{proof}

\section{Easily fixating regions: good blocks}\label{sec_good_blocks}

The aim of this section is to describe regions of $\mathds{Z}^2$, called \emph{good blocks} (Definition \ref{def_good_block}), favorable enough so that we can prove they will likely fixate at $+$, and to prove that the probability a region is favorable is very high (Proposition \ref{prop_good_blocks}). From now on, we consider both non-supercritical update families and supercritical update families with no disjoint update rules. For any $L \in \mathds{N}^*$, we denote $B_L=\{-L,...,L\}^2$. A \emph{block} is a set of the form $x+B_L$ where $x \in \mathds{Z}^2$. The definition of a good block will differ according to whether $\mathcal{U}$ is supercritical or not. If $\mathcal{U}$ is not supercritical, we set $M=\max\{\|x\|_2 \,|\, x \in D(1)\}$ and $M'=\max\{\|x\|_2 \,|\, x \in D'(1)\}$, where $D(1)$ and $D'(1)$ were defined at the beginning of Section \ref{sec_good_droplet}. For reasons that will be apparent later, we set $K=25$ if $\mathcal{U}$ is supercritical and $K=(2\lceil(4M+1)M'\rceil+1)^2$ if $\mathcal{U}$ is not supercritical.

\begin{definition}\label{def_good_block}
 Let $L \in \mathds{N}^*$. For any $x \in \mathds{Z}^2$, one says the block $x+B_L$ is \emph{good} when one of the following conditions is satisfied.
 \begin{itemize}
  \item If $\mathcal{U}$ is supercritical and contains no disjoint update rules, all sites in $x+B_L$ are infectable by the bootstrap percolation in $x+B_{2L}$ (with healthy boundary conditions) and initial infected sites the sites frozen at $+$ in $x+B_{2L}$, and there is no site frozen at $-$ in $x+B_{2L}$.
  \item If $\mathcal{U}$ is not supercritical, all sites in $x+D(3L)$ are infectable by the bootstrap percolation with initial infected sites the sites frozen at $+$ in $x+D(4L)$, there is no site frozen at $-$ in $x+D'(4LM+1)$, and there exists $2L \leq a \leq 3L$ such that $x+D(a)$ is good.
 \end{itemize}
\end{definition}

We remind the reader that the constants $a_0$, $\tilde a_0$ were defined in Proposition \ref{prop_key} and Lemma \ref{lem_technical}. In the following, when we say a quantity depends on $u_1,...,u_N$, this dependence will only apply for non-supercritical $\mathcal{U}$.

 \begin{proposition}\label{prop_good_blocks}
  There exists $\rho_0^-=\rho_0^-(\mathcal{U},u_1,...,u_N,\rho^+)>0$ and $L_0=L_0(\mathcal{U},u_1,...,u_N,\rho^+) \geq r$, with $L_0 \geq \max(a_0,\tilde a_0)$ if $\mathcal{U}$ is not supercritical, such that for any $0 \leq \rho^- \leq \rho_0^-$ and for any $x \in \mathds{Z}^2$, we have $\mathds{P}(x+B_{L_0}$ is not good$)\leq 1/2^{5K}$.
 \end{proposition}

 \begin{proof}
 We will deal differently with the non-supercritical case and with the supercritical case, though we need bootstrap percolation technology in both cases.
 
 \emph{Case $\mathcal{U}$ supercritical with no disjoint update rules.}
 
 To deal with this case, we will use the tools developed in \cite{Bollobas_et_al2015} for supercritical bootstrap percolation. Since $\mathcal{U}$ is supercritical, there exists an open semicircle of unstable directions. We denote its center by $u\in S^1$. We consider rectangles oriented in direction $u$: denoting $u^\perp \in S^1$ a direction orthogonal to $u$, for any $a,b>0$, the rectangle oriented in direction $u$ with width $a$ and length $b$ will be $R(a,b)=\{x \in \mathds{R^2} \,|\, -a/2 \leq \langle x,u^\perp \rangle \leq a/2, 0 < \langle x,u \rangle \leq b\}$. It was proven in Section 5.3 of \cite{Bollobas_et_al2015} that there exists some $a>0,b>0,c<+\infty$ depending only on $\mathcal{U}$ so that for any $x \in \mathds{Z}^2$, if $x+R(a,b)$ is initially infected, for any $b'>b$ the rectangle $x+R(a,b')$ is infectable by the bootstrap percolation in $x+R(a,b'+c)$ (the latter part is not explicitly stated in \cite{Bollobas_et_al2015}, but can be seen in the proofs). Consequently, if $L$ is large enough depending on $a,b,c$, for any $x \in \mathds{Z}^2$, $y \in x+B_L$, if there exists $1 \leq n \leq L/(2b)$ so that $y-nbu+R(a,b)$ is initially infected, then $y$ is infectable by the bootstrap percolation in $x+B_{2L}$ (with healthy boundary conditions). Therefore, if we denote $\mathcal{B}_y=\{y$ is not infectable by the bootstrap percolation in $x+B_{2L}$ with initial infected sites the sites frozen at $+$ in $x+B_{2L}\}$ and $k$ the maximum number of sites in a translation of $R(a,b)$, which is bounded, then $\mathds{P}(\mathcal{B}_y) \leq (1-(\rho^+)^k)^{\lfloor L/(2b) \rfloor}$. Thus $\mathds{P}(\cup_{y \in x+B_L}\mathcal{B}_y) \leq (2L+1)^2(1-(\rho^+)^k)^{\lfloor L/(2b) \rfloor}$. We choose $L_0 = L_0(\mathcal{U},u_1,...,u_N,\rho^+) \geq r$ so that $(2L_0+1)^2(1-(\rho^+)^k)^{\lfloor L_0/(2b) \rfloor}\leq 1/2^{5K+1}$, hence $\mathds{P}(\cup_{y \in x+B_{L_0}}\mathcal{B}_y) \leq 1/2^{5K+1}$. If we now set $\rho_0^- = 1/(2^{5K+1}(4L_0+1)^2)$, then for any $0 \leq \rho^- \leq \rho_0^-$, we have $\mathds{P}(x+B_{L_0}$ is not good$)\leq 1/2^{5K}$.

 \emph{Case $\mathcal{U}$ non-supercritical.} 
 
 Let $x\in\mathds{Z}^2$. We first study the probability of finding a good droplet $x+D(a)$. We assume $L \geq \tilde a_0$. Lemma \ref{lem_technical} yields that for $n \in \{1,...,\lfloor L/(r+1)\rfloor\}$, the $x+C_{i,i+1}(2L+n(r+1))$ are contained in $x+D(2L+n(r+1))$, and by their definition they are contained in $x+D(2L+(n-1)(r+1))^c$. This implies the events $\{x+D(2L+n(r+1))$ is good$\}$ for $n \in \{0,...,\lfloor L/(r+1)\rfloor\}$ depend on disjoint sets of sites hence are independent. Furthermore, we notice that if we set $i\in\{1,...,N\}$, for all $a>0$ the $C_{i,i+1}(a)$ are translations of each other. We denote $k_i$ the maximum number of sites in any translation of $C_{i,i+1}(1)$, which is finite, and $k=\sum_{i=1}^N k_i$. Then $\mathds{P}(\bigcap_{n=1}^{\lfloor L/(r+1)\rfloor}\{ x+D(2L+n(r+1))$ is not good$\}) \leq (1-(\rho^+)^k)^{\lfloor L/(r+1)\rfloor}$. If we choose $L_1=L_1(\mathcal{U},u_1,...,u_N,\rho^+) \geq \tilde a_0$ so that $(1-(\rho^+)^k)^{\lfloor L_1/(r+1)\rfloor} \leq 1/(3 \cdot 2^{5K})$, for all $L \geq L_1$ we have $\mathds{P}(\{\exists\, 2L \leq a \leq 3L, x+D(a)$ is good$\}^c) \leq 1/(3 \cdot 2^{5K})$. 
 
 We now consider the probability that sites are infectable. To deal with it, we first notice that $\rho^+ > \tilde q_c(\mathcal{U})$ (indeed, Theorem 3.1 of \cite{Hartarsky2018} states $\tilde q_c(\mathcal{U})=0$ for critical update families). This allows to use Theorem 3.5 of \cite{Hartarsky2018}, which states that since $\rho^+ > \tilde q_c(\mathcal{U})$, there exists a constant $c=c(\mathcal{U},\rho^+)$ so that for any $n \in \mathds{N}$, we have $\mathds{P}(0$ is not infectable by the bootstrap percolation starting from sites frozen at $+$ in $B_n) \leq e^{-cn}$. Moreover, for $y \in x+D(3L)$, we have $y+B_{\lfloor L/\sqrt{2}\rfloor} \subset x+D(4L)$, hence $\mathds{P}(y$ is not infectable starting from sites frozen at $+$ in $x+D(4L)) \leq \mathds{P}(y$ is not infectable starting from sites frozen at $+$ in $y+B_{\lfloor L/\sqrt{2}\rfloor}) \leq e^{-c\lfloor L/\sqrt{2}\rfloor}$. In addition, remembering the definition of $M,M'$ at the beginning of the section, there are at most $(6ML+1)^2$ sites in $x+D(3L)$, hence $\mathds{P}(\bigcup_{y \in x+D(3L)}\{y$ is not infectable starting from sites frozen at $+$ in $x+D(4L)\})\leq (6ML+1)^2 e^{-c\lfloor L/\sqrt{2}\rfloor}$. We then set $L_0=L_0(\mathcal{U},u_1,...,u_N) \geq \max(L_1,a_0,1)$ so that $(6ML_0+1)^2 e^{-c\lfloor L_0/\sqrt{2}\rfloor} \leq 1/(3 \cdot 2^{5K})$. If we now set $\rho_0^- = \frac{1}{3 \cdot 2^{5K}(2(4L_0M+1)+1)^2}$, then for any $0 \leq \rho^- \leq \rho_0^-$, we have $\mathds{P}(x+B_{L_0}$ is not good$)\leq 1/2^{5K}$.
 \end{proof}

 \begin{remark}
  It is in the proof of Proposition \ref{prop_good_blocks} that we need $\rho^+ > \tilde q_c(\mathcal{U})$ instead of $\rho^+ > q_c(\mathcal{U})$ for subcritical dynamics.
 \end{remark}
 
 \section{Fixation at $+$: proof of Theorems \ref{thm_fixation} and \ref{thm_connected_components}}\label{sec_fixation}

 The goal of this section is to prove Theorem \ref{thm_fixation}, which states fixation at $+$ of all sites occurs if $\rho^-=0$, and Theorem \ref{thm_connected_components}, which claims the existence of a deterministic time at which the connected components of sites not fixated at $+$ are finite. In order to do that, we begin by showing Theorem \ref{thm_connected_components_blocks}, which roughly states that for any $0 \leq \rho^- \leq \rho_0^-$ (with $\rho_0^-$ the one defined in Proposition \ref{prop_good_blocks}), there exists a deterministic time so that the connected components of blocks containing sites not well fixed at this time are finite. This result is stronger than Theorem \ref{thm_connected_components}, and is central in the proof of Theorems \ref{thm_fixation} and \ref{thm_flippers}, since it allows to consider only the dynamics in these finite connected components. At the end of this section, we prove Theorem~\ref{thm_fixation}. To state Theorem \ref{thm_connected_components_blocks}, we consider the blocks $2L_0 x +B_{L_0}$ for all $x \in \mathds{Z}^2$ (with $L_0$ the one defined in Proposition \ref{prop_good_blocks}), denoting $B_{L_0}(x)=2L_0 x +B_{L_0}$. Blocks $B_{L_0}(x)$ and $B_{L_0}(y)$ will be considered neighbors when $\|x-y\|_\infty=1$, which gives a graph with an associated notion of connectedness. We then have the following.

 \begin{theorem}\label{thm_connected_components_blocks}
  There exists $T_0 = T_0(\mathcal{U},u_1,...,u_N,\rho^+) < +\infty$ deterministic so that for any $0 \leq \rho^- \leq \rho_0^-$, for any choice of initial distribution $\mu$, almost surely the connected components of blocks containing sites that are not well fixed at $+$ at time $T_0$ are finite.
 \end{theorem}

 \begin{proof}
 We first give the idea of the proof. We will lower bound the probability that the dynamics fills a good block with $+$ in a time interval of length 1, which allows to find $T_0$ so that for any good block, the probability that the dynamics fills the block with $+$ before time $T_0$ is very high. We also show that roughly, if the sites of a good block are at $+$ at some time, they are well fixed at $+$ at this time, which will be Claim \ref{claim_conn_blocks}.  Since by Proposition \ref{prop_good_blocks} blocks are likely to be good, the probability that a given block is full of sites well fixed at $+$ at time $T_0$ is then close to 1 (to shorten the notation, for any $T>0$, we say a block is \emph{$T$-fixed} when all its sites are well fixed at $+$ at time $T$). We then want to use the following classical percolation argument: if there is an infinite connected component of non $T_0$-fixed blocks, for any integer $n$ there is a path of length $n$ of such blocks, which has probability tending to 0 when $n$ tends to $+\infty$. However, one has to be careful in proving the latter part, since the events that the blocks are $T_0$-fixed are not independent.
 
 We now give the rigorous argument. We set $0 \leq \rho^- \leq \rho_0^-$ and $\mu$ an initial distribution. It is enough to find $T_0$ so that for any $x \in \mathds{Z}^2$, almost surely the connected component of non $T_0$-fixed blocks containing $B_{L_0}(x)$ is finite. Let $x \in \mathds{Z}^2$, $T>0$. If $B_{L_0}(x)$ is contained in an infinite connected component of non $T$-fixed blocks, for all $n>0$ there exists a \emph{non $T$-fixed path of legnth $n$ starting from $B_{L_0}(x)$}, that is a sequence $B_{L_0}(x)=B_{L_0}(x_1),...,B_{L_0}(x_n)$ of non $T$-fixed blocks such that the $x_i$, $i\in\{1,...,n\}$ are all different and for each $i\in\{1,...,n-1\}$, $B_{L_0}(x_{i+1})$ is a neighbor of $B_{L_0}(x_i)$. There are at most $8^n$ possible such paths, hence it is enough to find $T_0$ so that $8^n\max_\gamma\mathds{P}(\gamma$ is non $T_0$-fixed$)$ tends to 0 when $n$ tends to $+\infty$, where the max is taken over all possible paths of length $n$ starting from $B_{L_0}(x)$ and all $x \in \mathds{Z}^2$.
 
 Let $\gamma$ be a path of length $n$ starting from $B_{L_0}(x)$. We will study $\max_\gamma\mathds{P}(\gamma$ is non $T_0$-fixed$)$. Since the events that two neighboring blocks are good are not independent, we will define a subset of blocks in $\gamma$ which are sufficiently far apart to get independence. Remembering the definition of $K,M,M'$ at the beginning of Section \ref{sec_good_blocks}, if $\mathcal{U}$ is supercritical, we can find blocks $B_{L_0}(y_1),...,B_{L_0}(y_{\lfloor n/K \rfloor})$ in $\gamma$ so that for $i \neq j$ we have $2L_0y_i+B_{2L_0}$ and $2L_0y_j+B_{2L_0}$ disjoint, and if $\mathcal{U}$ is not supercritical, we can find blocks $B_{L_0}(y_1),...,B_{L_0}(y_{\lfloor n/K \rfloor})$ in $\gamma$ so that for $i \neq j$ we have $\|y_i-y_j\|_\infty>\frac{(4L_0M+1)M'}{L_0}$ hence $2L_0y_i+D'(4L_0M+1)$ and $2L_0y_j+D'(4L_0M+1)$ are disjoint. In order to have the same notation in the supercritical case and the non-supercritical case, we denote $B_i=2L_0y_i+B_{L_0}$, $B_i'=2L_0y_i+B_{2L_0}$ if $\mathcal{U}$ is supercritical, and $B_i=2L_0y_i+D(3L_0)$, $B_i'=2L_0y_i+D'(4L_0M+1)$ if $\mathcal{U}$ is not supercritical. 
 
 For each $i \in \{1,...,\lfloor n/K \rfloor\}$, assuming $B_{L_0}(y_i)$ is good, we will define events $\mathcal{G}_{i,\ell}$, $\ell \in \mathds{N}$ such that all sites in $B_{L_0}(y_i)$ (and some more) are at $+$ at time $\ell$ if $\mathcal{G}_{i,\ell}$ occurs. These events will be of the form ``a sequence of clock rings and choices of update rules happens that successively sets to $+$ all sites that were at $-$''. We now construct this sequence. We assume $B_{L_0}(y_i)$ is good. Then we can prove that all sites in $B_i$ are infectable by the bootstrap percolation in $B_i'$ with healthy boundary conditions and initial infected sites the sites frozen at $+$ in $B_i'$. Indeed, if $\mathcal{U}$ is supercritical this comes from the definition of a good block. If $\mathcal{U}$ is not supercritical, this is because sites in $B_i$ are infectable by the bootstrap percolation starting from sites frozen at $+$ in $2L_0y_i+D(4L_0) \subset B_i'$, hence starting from sites frozen at $+$ in $B_i'$, and in the bootstrap percolation starting from sites frozen at $+$ in $B_i'$, all sites outside $B_i'$ remain healthy since $2L_0y_i+D'(4L_0M+1)$ is constructed so the infection cannot escape it. Consequently, in both cases there exists $m_i \in \mathds{N}$, a sequence $x_{i,1},...,x_{i,m_i}$ of distinct sites in $B_i'$ and a sequence $X_{i,1},...,X_{i,m_i}$ of update rules so that all sites in $B_i$ are either equal to one of the $x_{i,j}$ or frozen at $+$, for each $j\in\{1,...,m_i\}$, $x_{i,j}+X_{i,j} \subset B_i'$, and all sites of $x_{i,j}+X_{i,j}$ are either some $x_{i,j'}$ with $j'<j$, or frozen at $+$. For each $\ell\in\mathds{N}^*$, we denote $\mathcal{G}_{i,\ell}$ the event ``in the time interval $(\ell-1,\ell]$, there are successive clock rings at each $x_{i,j}$ that is at $-$ at time $\ell-1$, in increasing order, the update rule chosen for $x_{i,j}$ is $X_{i,j}$, and no other clock ring happens in $B_i'$''. If $\mathcal{G}_{i,\ell}$ happens, sites $x_{i,1},...,x_{i,m_i}$ are at $+$ at time $\ell$ (this uses the fact that none of them is frozen at $-$). 
 Moreover, we have the following, thanks to which if $\mathcal{G}_{i,\ell}$ happens, $B_{L_0}(y_i)$ is $\ell$-fixed. 
 
 \begin{claim}\label{claim_conn_blocks}
 If $x_{i,1},...,x_{i,m_i}$ are at $+$ at time $\ell$ then all sites in $B_{L_0}(y_i)$ are well fixed at $+$ at time $\ell$. 
 \end{claim}
 
 \begin{proof}
 We first give the idea of the proof. If $\mathcal{U}$ is not supercritical, then since $B_{L_0}(y_i)$ is good, it is contained in some $2L_0y_i+D(a)$ which is good, and if $x_{i,1},...,x_{i,m_i}$ are at $+$ at time $\ell$ this $2L_0y_i+D(a)$ will be full of $+$, so by Proposition \ref{prop_key} all sites in $2L_0y_i+D(a)$ are well fixed at $+$. If $\mathcal{U}$ is supercritical and contains no disjoint update rules, if $x_{i,1},...,x_{i,m_i}$ are at $+$ at time $\ell$, then we will see no $x_{i,j}+X$, $X\in \mathcal{U}$ can be entirely at $-$, so no $x_{i,j}$ can switch to $-$. 
 
 We now write the rigorous argument. If $\mathcal{U}$ is not supercritical, if $x_{i,1},...,x_{i,m_i}$ are at $+$ at time $\ell$, then all sites in $2L_0y_i+D(3L_0)$ are at $+$ at time $\ell$. Moreover, there exists $2L_0 \leq a \leq 3L_0$ such that $2L_0y_i+D(a)$ is good. By Proposition \ref{prop_key} and the invariance by translation of the dynamics, all sites in $2L_0y_i+D(a)$ are well fixed at $+$ at time $\ell$, hence all sites in $B_{L_0}(y_i)$ are well fixed at $+$ at time $\ell$. If $\mathcal{U}$ is supercritical and contains no disjoint update rules, we assume by contradiction that $x_{i,1},...,x_{i,m_i}$ are at $+$ at time $\ell$ but some sites among $x_{i,1},...,x_{i,m_i}$ are not well fixed at $+$ at this time. Then there exists an $x_{i,j}$ that can be the first of them to change its state to $-$ after time $\ell$. Sites in $x_{i,j}+X_{i,j}$ are either frozen at $+$ or some $x_{i,j'}, j' <j$, hence are at $+$ at the time of the switch. But $\mathcal{U}$ has no disjoint update rules, thus for any $X \in \mathcal{U}$, $x_{i,j}+X$ contains at least a site at $+$ at the time of the switch, therefore $x_{i,j}$ cannot change its state to $-$. We deduce that if $x_{i,1},...,x_{i,m_i}$ are at $+$ at time $\ell$, then $x_{i,1},...,x_{i,m_i}$ are well fixed at $+$ at time $\ell$, hence all sites in $B_{L_0}(y_i)$ are well fixed at $+$ at time $\ell$. 
 \end{proof}
 
 We now study $\mathds{P}(\mathcal{G}_{i,\ell}|\mathcal{F}_{\ell-1})$ for any $i\in\{1,...,\lfloor n/K\rfloor\},\ell \in \mathds{N}^*$, and give a lower bound on it that is uniform on $i$ and $\ell$. We denote $b'$ the number of sites in $B_i'$ (which does not depend on $i$). If $B_{L_0}(y_i)$ is good then conditionally on $\mathcal{F}_{\ell-1}$, the probability that $\mathcal{G}_{i,\ell}$ occurs is the probability that in a time interval of length 1, on a given set of sites of cardinal at most $b'$, there are successive clock rings, on another given set of sites of cardinal at most $b'$ there are no clock rings, and when making a given number of update rules choices, smaller than $b'$, the results follow a given sequence. This is bigger than the probability of having successive clock rings on $b'$ sites and no clock rings on $b'$ other sites in a time interval of length 1, and obtaining a given sequence in $b'$ update rules choices. Therefore there exists $\varepsilon=\varepsilon(\mathcal{U},u_1,...,u_N,\rho^+)>0$ so that for any $i\in\{1,...,\lfloor n/K\rfloor\},\ell \in \mathds{N}^*$, we have $\mathds{1}_{\{B_{L_0}(y_i)\text{ is good}\}}\mathds{P}(\mathcal{G}_{i,\ell}|\mathcal{F}_{\ell-1}) \geq \mathds{1}_{\{B_{L_0}(y_i)\text{ is good}\}}\varepsilon$.
 
 We are now able to bound the probability that $\gamma$ is non $T$-fixed for $T \in \mathds{N}^*$. Indeed, if $\gamma$ is non $T$-fixed, for all $i\in\{1,...,\lfloor n/K\rfloor\}$, $1 \leq \ell \leq T$ so that $B_{L_0}(y_i)$ is good, $(\mathcal{G}_{i,\ell})^c$ occurs. We are going to separate according to the set of blocks $B_{L_0}(y_i)$ which are good. We denote $G_\gamma = \{i\in\{1,...,\lfloor n/K\rfloor\} \,|\, B_{L_0}(y_i)$ is good$\}$, and for $G\subset\{1,...,\lfloor n/K\rfloor\}$ we study $\mathds{P}(G_\gamma=G,\cap_{i\in G,1 \leq \ell \leq T}(\mathcal{G}_{i,\ell})^c)$. In addition, $\mathcal{G}_{i,\ell}$ depends on clock rings and update rules choices in $B_i'$, and the definition of the $x_i$ yields that the $B_i'$ are disjoint, hence the $\mathcal{G}_{i,\ell}$ are independent conditionally on $\mathcal{F}_{\ell-1}$. Consequently, $\mathds{P}(G_\gamma=G,\cap_{i\in G,1 \leq \ell \leq T}(\mathcal{G}_{i,\ell})^c) \leq (1-\varepsilon)^{|G|}\mathds{P}(G_\gamma=G,\cap_{i\in G,1 \leq \ell \leq T-1}(\mathcal{G}_{i,\ell})^c)$, hence $\mathds{P}(G_\gamma=G,\cap_{i\in G,1 \leq \ell \leq T}(\mathcal{G}_{i,\ell})^c) \leq (1-\varepsilon)^{|G|T}\mathds{P}(G_\gamma=G)$. We now estimate $\mathds{P}(G_\gamma=G)$. For any $i\in\{1,...,\lfloor n/K\rfloor\}$, the event $\{B_{L_0}(y_i)$ is good$\}$ depends only on the frozen sites inside $B_i'$ (this requires Lemma \ref{lem_technical}), and the $B_i'$ are disjoint, so these events are independent. By Proposition \ref{prop_good_blocks}, we deduce $\mathds{P}(G_\gamma=G,\cap_{i\in G,1 \leq \ell \leq T}(\mathcal{G}_{i,\ell})^c) \leq (1-\varepsilon)^{|G|T}(1/2^{5K})^{\lfloor n/K\rfloor-|G|}$. If we choose $T_0=T_0(\mathcal{U},u_1,...,u_N,\rho^+) < +\infty$ so that $(1-\varepsilon)^{T_0} \leq 1/2^{5K}$, we obtain $\mathds{P}(G_\gamma=G,\cap_{i\in G,1 \leq \ell \leq T_0}(\mathcal{G}_{i,\ell})^c) \leq (1/2^{5K})^{\lfloor n/K\rfloor}$. Since there are at most $2^n$ choices for $G$, this implies $8^n\mathds{P}(\gamma$ is non $T_0$-fixed$) \leq 8^n2^n(1/2^{5K})^{\lfloor n/K\rfloor}=2^{4n}(1/2^{5K})^{\lfloor n/K\rfloor}$, which decays to 0 when $n$ tends to $+\infty$. This ends the proof of Theorem \ref{thm_connected_components_blocks}.
 \end{proof}

 We are now in position to prove Theorem \ref{thm_fixation}.
 
 \begin{proof}[Proof of Theorem \ref{thm_fixation}.] 
  The idea behind the argument is that Theorem \ref{thm_connected_components_blocks} shows that at time $T_0$ there are finite connected components of sites that have not fixated at $+$ in an ocean of sites well fixed at $+$. When by chance the dynamics fills one of these connected components with $+$, no site at $-$ can appear near it, so all sites in the component remain at $+$ forever. We now give the rigorous argument. Let $\rho^-=0$, and $\mu$ be an initial distribution. Let $x \in \mathds{Z}^2$. Thanks to Theorem \ref{thm_connected_components_blocks}, if $x$ is not well fixed at $+$ at time $T_0$, almost surely $x$ is contained in a block that is part of a finite connected component $\mathcal{C}$ of blocks containing sites not well fixed at $+$ at time $T_0$. Then the blocks in $\mathcal{C}^c$ that are neighbors of elements of $\mathcal{C}$ contain only sites well fixed at $+$ at time $T_0$. Moreover, Theorem 1 of \cite{Balister_et_al2016} states that if all directions are stable, then $\mathcal{U}$ is subcritical and $q_c(\mathcal{U})=1$. This implies $\tilde q_c(\mathcal{U})=1$, so one cannot choose $\tilde q_c(\mathcal{U}) < \rho^+ < 1$. Therefore in our case there exists an unstable direction $u$, and some $X \in \mathcal{U}$ with $X \subset \mathds{H}_u$. We enumerate the sites in $\bigcup_{B \in \mathcal{C}} B$ as $x_1,...,x_n$ with $\langle x_i,u \rangle$ nondecreasing. Then, similarly to what was done in the proof of Theorem \ref{thm_connected_components_blocks}, for each $\ell \in \mathds{N}^*$, we define $\mathcal{G}_{\ell}$ the event ``in the time interval $(T_0+\ell-1,T_0+\ell]$, there are successive clock rings at each $x_{i}$ that is at $-$ at time $T_0+\ell-1$, in increasing order, the update rule chosen for $x_{i}$ is $X$, and no other clock ring happens in $\bigcup_{B \in \mathcal{C}} B$''. If $\mathcal{G}_\ell$ occurs, all sites of $\bigcup_{B \in \mathcal{C}} B$ are at $+$ at time $T_0+\ell$. They have then fixated at $+$ at time $T_0+\ell$, as the first site $y \in \bigcup_{B \in \mathcal{C}} B$ to switch to state $-$ would need $X'\in\mathcal{U}$ so that the sites of $y+X'$ are at state $-$. Furthermore, as in the proof of Theorem \ref{thm_connected_components_blocks}, there exists $\varepsilon>0$ random $\mathcal{F}_{T_0}$-measurable (which depends on $|\bigcup_{B \in \mathcal{C}} B|$) so that for all $\ell \in \mathds{N}^*$ we have $\mathds{1}_{\{x\text{ not well fixed at }T_0\}}\mathds{P}(\mathcal{G}_\ell|\mathcal{F}_{T_0+\ell-1}) \geq \mathds{1}_{\{x\text{ not well fixed at }T_0\}}\varepsilon$, thus for all $n \in \mathds{N}^*$ we obtain $\mathds{P}(x$ not well fixed at $T_0, \bigcap_{\ell=1}^n \mathcal{G}_\ell^c) \leq \mathds{E}((1-\varepsilon)^n)$, which converges to 0 when $n$ tends to $+\infty$ by dominated convergence, thus almost surely one of the $\mathcal{G}_\ell$ occurs. This implies $x$ almost surely fixates at $+$, which ends the proof of Theorem \ref{thm_fixation}.
 \end{proof}

 \section{Flippers: proof of Theorem \ref{thm_flippers}}\label{sec_flippers}
 
 This section is devoted to the proof of Theorem \ref{thm_flippers} on the existence or non-existence of flippers. Let $0 < \rho^- \leq \rho_0^-$ (with $\rho_0^-$ the one defined in Proposition \ref{prop_good_blocks}), and $\mu$ be an initial distribution. 
 
 \subsection{First case: $\mathcal{U}$ contains two disjoint update rules $X$ and $X'$.} Almost surely, for each $x\in \mathds{Z}^2$, $t \geq 0$, there is a clock ring at $x$ after time $t$ so that $X$ is chosen and a clock ring at $x$ after time $t$ so that $X'$ is chosen. Moreover, almost surely there is an infinite number of sites $x$ so that all sites in $x+X$ are frozen at $+$ and all sites in $x+X'$ are frozen at $-$. If $x$ is one of these sites, then $x$ is a flipper. Hence almost surely there is an infinite number of flippers.
 
 \subsection{Second case: $\mathcal{U}$ contains no disjoint update rules.} We will show that almost surely there is no flipper. The idea of the proof is to use Theorem \ref{thm_connected_components_blocks}, which states that the connected components of blocks with sites not well fixed at $+$ at time $T_0$ are finite, and to restrict our attention to one of these finite components. We will then prove that in this component, at some time all the possible flippers successively switch from $-$ to $+$. Then, if $x$ is one of these sites, since it was put to $+$, some $x+X$ with $X \in \mathcal{U}$ is at $+$. Then for any $X'\in \mathcal{U}$, since $X$ and $X'$ are not disjoint, $x+X'$ contains at least a site of $x+X$, so $x+X'$ is not entirely at $-$, thus $x$ cannot be set at $-$ hence cannot be a flipper.
 
 We now spell out the rigorous argument. Let $x \in \mathds{Z}^2$. If $x$ is well fixed at time $T_0$, then $x$ is not a flipper. Thanks to Theorem \ref{thm_connected_components_blocks}, if $x$ is not well fixed at $+$ at time $T_0$, almost surely $x$ is contained in a block that is part of a finite connected component $\mathcal{C}$ of blocks containing sites not well fixed at $+$ at time $T_0$. We denote $\mathcal{C}^\to$ the set of blocks in $\mathcal{C}^c$ that are neighbors of elements of $\mathcal{C}$. Then $\bigcup_{B \in \mathcal{C}^\to}B$ contain only sites well fixed at $+$ at time $T_0$. In this proof, we will need the concept of $\ominus$-bootstrap percolation and $\oplus$-bootstrap percolation. The $\ominus$-bootstrap percolation (respectively $\oplus$-bootstrap percolation) has the same dynamics as the usual bootstrap percolation, apart that the sites frozen at $+$ (respectively at $-$) in the $\mathcal{U}$-voter dynamics are frozen in the healthy state.
 
 \begin{claim}\label{claim_flippers}
  For any $t \geq T_0$, the $\ominus$-bootstrap percolation process starting from the sites of $\bigcup_{B \in \mathcal{C}}B$ that are at $-$ at time $t$ infects only sites in $\bigcup_{B \in \mathcal{C}}B$.
 \end{claim}

 \begin{proof}
  The idea is that $\bigcup_{B \in \mathcal{C}}B$ is surrounded by $\bigcup_{B \in \mathcal{C}^\to}B$ whose sites are well fixed at $+$, so cannot be infected by $-$, hence the infection cannot escape $\bigcup_{B \in \mathcal{C}}B$. We show this by contradiction. If the claim does not hold, let $x'$ be a site of $(\bigcup_{B \in \mathcal{C}}B)^c$ infected by this $\ominus$-bootstrap percolation process at the first step at which it is possible, then $x' \in \bigcup_{B \in \mathcal{C}^\to}B$, so $x'$ is well fixed at $+$ at time $T_0$. However, the infectability of $x'$ will imply the contrary. Indeed, it implies the existence of a sequence $x_0,...,x_n=x'$ and $X_1,...,X_n \in \mathcal{U}$ so that the $x_j$, $j \in \{1,...,n\}$ are at $+$ at time $t$ but not frozen at $+$, and for any $i \in \{1,...,n\}$, the sites of $x_i+X_i$ are among the $x_j$, $j < i$ or sites of $\bigcup_{B \in \mathcal{C}}B$ at $-$ at time $t$. If after time $t$ there are successive clock rings at $x_1,...,x_n$, if the update rules chosen are $X_1,...,X_n$, and if there is no other clock ring in any of the $x_j$, $j \in \{1,...,n\}$ or in $\bigcup_{B \in \mathcal{C}}B$ during this time, which is a possible event, then $x'$ is at $-$ afterwards, which contradicts the fact it is well fixed at $+$ at time $T_0$ and proves the claim. 
  \end{proof}
 
 We now define events $\mathcal{G}_\ell$, $\ell \in \mathds{N}^*$ so that if one of the $\mathcal{G}_\ell$ occurs, there is no flipper in $\bigcup_{B \in \mathcal{C}}B$. These events are of the form ``a first sequence of clock rings and choices of update rules happens so that all sites of $\bigcup_{B \in \mathcal{C}}B$ that can switch to $-$ do so, then a second sequence so that all sites of $\bigcup_{B \in \mathcal{C}}B$ that can switch to $+$ do so''. We begin by constructing the first sequence. For any $\ell \in \mathds{N}^*$, we consider all the sites that are at $+$ at time $T_0+\ell-1$ and can be infected by the $\ominus$-bootstrap percolation starting from the sites of $\bigcup_{B \in \mathcal{C}}B$ that are at $-$ at time $T_0+\ell-1$. All these sites are in $\bigcup_{B \in \mathcal{C}}B$ by Claim \ref{claim_flippers}. We denote them $x_{\ell,1,1},...,x_{\ell,1,k_\ell^1}$, ordered so that for any $i \in \{1,...,k_\ell^1\}$, there exists $X_{\ell,1,i}\in\mathcal{U}$ so that the sites of $x_{\ell,1,i}+X_{\ell,1,i}$ are among the $x_{\ell,1,j}$, $j < i$ or sites of $\bigcup_{B \in \mathcal{C}}B$ at $-$ at time $T_0+\ell-1$. We then define $\mathcal{G}_{\ell,1}$ as the event ``in the time interval $(T_0+\ell-1,T_0+\ell-1/2]$, there are successive clock rings at $x_{\ell,1,1},...,x_{\ell,1,k_\ell^1}$ in increasing order, the update rules that are chosen are $X_{\ell,1,1},...,X_{\ell,1,k_\ell^1}$, and there is no other clock ring in $\bigcup_{B \in \mathcal{C}}B$''. We now construct a second sequence of clock rings and choices of update rules, which will allow all sites of $\bigcup_{B \in \mathcal{C}}B$ that can switch to $+$ after time $T_0+\ell-1/2$ to do so. To do that, we consider all the sites of $\bigcup_{B \in \mathcal{C}}B$ that are at $-$ at time $T_0+\ell-1/2$ and can be infected by the $\oplus$-bootstrap percolation in $\bigcup_{B \in \mathcal{C} \cup \mathcal{C}^\to}B$ starting from the sites of $\bigcup_{B \in \mathcal{C} \cup \mathcal{C}^\to}B$ that are at $+$ at time $T_0+\ell-1/2$. We denote them $x_{\ell,2,1},...,x_{\ell,2,k_\ell^2}$, ordered so that for any $i \in \{1,...,k_\ell^2\}$, there exists $X_{\ell,2,i}\in\mathcal{U}$ so that the sites of $x_{\ell,2,i}+X_{\ell,2,i}$ are among the $x_{\ell,2,j}$, $j < i$ or sites of $\bigcup_{B \in \mathcal{C} \cup \mathcal{C}^\to}B$ at $+$ at time $T_0+\ell-1/2$. We then define $\mathcal{G}_{\ell,2}$ as the event ``in the time interval $(T_0+\ell-1/2,T_0+\ell]$, there are successive clock rings at $x_{\ell,2,1},...,x_{\ell,2,k_\ell^2}$ in increasing order, the update rules that are chosen are $X_{\ell,2,1},...,X_{\ell,2,k_\ell^2}$, and there is no other clock ring in $\bigcup_{B \in \mathcal{C}}B$''. We now set $\mathcal{G}_\ell = \mathcal{G}_{\ell,1} \cap \mathcal{G}_{\ell,2}$. We will show the following.
 
 \begin{claim}
 For any $\ell\in\mathds{N}^*$, if $\mathcal{G}_\ell$ occurs, almost surely there is no flipper in $\bigcup_{B \in \mathcal{C}}B$.
 \end{claim}
 
 \begin{proof}
 The idea of the proof is that if $x'\in\bigcup_{B \in \mathcal{C}}B$ is a flipper, then it can switch to $-$ and to $+$, so if $\mathcal{G}_\ell$ occurs, $x'$ is at $-$ at time $T_0+\ell-1/2$ and is then set at $+$ before time $T_0+\ell$, so there is some $X\in\mathcal{U}$ so that $x'+X$ is at $+$, and since there are no disjoint update rules, for any $X'\in\mathcal{U}$, the set $x'+X'$ contains a site of $x'+X$ hence is not entirely at $-$, so $x'$ actually cannot switch to $-$ later, hence is not a flipper. 
 
 We now give the rigorous argument, beginning by formalizing ``if $x\in\bigcup_{B \in \mathcal{C}}B$ is a flipper, then it can switch to $-$ and to $+$''. We show by induction that any site of $\bigcup_{B \in \mathcal{C}}B$ that changes its state to $-$ after time $T_0+\ell-1$ is infectable by the $\ominus$-bootstrap percolation process starting from the sites of $\bigcup_{B \in \mathcal{C}}B$ that are at $-$ at time $T_0+\ell-1$. Indeed, if $x'$ is the $n$-th site to do so, there exists $X \in \mathcal{U}$ so that all sites in $x'+X$ are at $-$ just before the change. Since sites in $\bigcup_{B \in \mathcal{C}^\to}B$ are well fixed at $+$ at time $T_0$, $x'+X \subset \bigcup_{B \in \mathcal{C}}B$, and the sites of $x'+X$ were either at $-$ at time $T_0+\ell-1$ or changed their state to $-$ after this time. In both cases, by the induction hypothesis they are infectable by the $\ominus$-bootstrap percolation starting from the sites of $\bigcup_{B \in \mathcal{C}}B$ that are at $-$ at time $T_0+\ell-1$, hence $x'$ also is, since it is not frozen at $+$. A similar argument yields that any site in $\bigcup_{B \in \mathcal{C}}B$ that changes its state to $+$ after time $T_0+\ell-1/2$ is infectable by the $\oplus$-bootstrap percolation in $\bigcup_{B \in \mathcal{C} \cup \mathcal{C}^\to}B$ starting from the sites of $\bigcup_{B \in \mathcal{C} \cup \mathcal{C}^\to}B$ that are at $+$ at time $T_0+\ell-1/2$. Therefore if $x' \in \bigcup_{B \in \mathcal{C}}B$ is a flipper, it is infectable by these two processes. 
 
 We now assume $\mathcal{G}_\ell$ occurs and prove that almost surely there is no flipper in $\bigcup_{B \in \mathcal{C}}B$. We assume by contradiction that $x' \in \bigcup_{B \in \mathcal{C}}B$ is a flipper. Then $\mathcal{G}_{\ell,1}$ ensures $x'$ is at $-$ at time $T_0+\ell-1/2$, and $\mathcal{G}_{\ell,2}$ ensures $x'$ switches to $+$ between times $T_0+\ell-1/2$ and $T_0+\ell$. Since $x'$ is a flipper, it will switch to $-$ after time $T_0+\ell$. Let $x''$ be the first site to switch to $-$ at some time $t \geq T_0+\ell$ among the sites of $\bigcup_{B \in \mathcal{C}}B$ that switch to $+$ between times $T_0+\ell-1/2$ and $T_0+\ell$. There exists $X \in \mathcal{U}$ so that the sites in $x''+X$ were at $+$ at the time of this switch to $+$. Since $x''+X \subset \bigcup_{B \in \mathcal{C} \cup \mathcal{C}^\to}B$, since sites in $\bigcup_{B \in \mathcal{C}^\to}B$ are well fixed at time $T_0$, and by the structure of $\mathcal{G}_{2,\ell}$, sites in $x''+X$ are still at $+$ at time $T_0+\ell$. They are still at $+$ at time $t$, since if one of them had switched to $-$ before, it would have been infectable by the $\ominus$-bootstrap percolation starting from the sites of $\bigcup_{B \in \mathcal{C}}B$ that are at $-$ at time $T_0+\ell-1$, thus since $\mathcal{G}_{\ell,1}$ occurs it would be at $-$ at time $T_0+\ell-1/2$, thus it would have switched at $+$ between times $T_0+\ell-1/2$ and $T_0+\ell$, and $x''$ is the first such site to switch to $-$ after time $T_0+\ell$. This implies sites of $x''+X$ are at $+$ at time $t$. Moreover, we assumed $\mathcal{U}$ contains no disjoint update rules, so for any $X' \in \mathcal{U}$, we have that $x''+X'$ contains a site of $x''+X$, which is at $+$ at time $t$. Hence $x''$ cannot switch its state to $-$ at time $t$, so there is a contradiction, which ends the proof of the claim. 
 \end{proof}
 
 We conclude that for any $\ell\in \mathds{N}^*$, if $\mathcal{G}_\ell$ occurs, $x$ is not a flipper. Consequently, it is enough to prove that almost surely, if $x$ is not well fixed at time $T_0$, one of the $\mathcal{G}_\ell$, $\ell \in \mathds{N}^*$ occurs. Furthermore, as in the proof of Theorem \ref{thm_connected_components_blocks}, there exists $\varepsilon>0$ random $\mathcal{F}_{T_0}$-measurable (which depends on $|\bigcup_{B \in \mathcal{C}} B|$) so that for all $\ell \in \mathds{N}^*$ we have $\mathds{1}_{\{x\text{ not well fixed at }T_0\}}\mathds{P}(\mathcal{G}_\ell|\mathcal{F}_{T_0+\ell-1}) \geq \mathds{1}_{\{x\text{ not well fixed at }T_0\}}\varepsilon$, thus for all $n \in \mathds{N}^*$ we obtain $\mathds{P}(x$ not well fixed at $T_0,\bigcap_{\ell=1}^n \mathcal{G}_\ell^c) \leq \mathds{E}((1-\varepsilon)^n)$, which converges to 0 when $n$ tends to $+\infty$ by dominated convergence, thus almost surely if $x$ is not well fixed at time $T_0$ one of the $\mathcal{G}_\ell$ occurs, which ends the proof of Theorem \ref{thm_flippers}. 

\section{One-dimensional case: proof of Theorems \ref{thm_fixation_dim1} and \ref{thm_flippers_dim1}}\label{sec_dim1}

This section is devoted to the proof of Theorems \ref{thm_fixation_dim1} and \ref{thm_flippers_dim1} on one-dimensional update families. In this section, $\mathcal{U}$ will be an update family on $\mathds{Z}$. As in the two-dimensional case, we define the range of $\mathcal{U}$ as $r=r(\mathcal{U})=\max\{|x| \,|\, x \in X, X \in \mathcal{U}\}$. Let $0 < \rho^+ < 1$. The following obvious fact will be key to our proofs, replacing the two-dimensional Theorem \ref{thm_connected_components_blocks}.

\begin{lemma}
 For any $0 \leq \rho^- < 1-\rho^+$, for any $x\in \mathds{Z}$, almost surely there exists $x_r>x$ and $x_\ell<x$ so that $x_r+1,...,x_r+r$ and $x_\ell-1,...,x_\ell-r$ are frozen at $+$.
\end{lemma}

\begin{proof}[Proof of Theorem \ref{thm_fixation_dim1}.]
We want to prove fixation at $+$ for all sites of $\mathds{Z}$. The argument resembles the one used for the two-dimensional case: the dynamics in $\{x_\ell,...,x_r\}$ is isolated from what happens outside by the sites frozen at $+$, and once by chance $\{x_\ell,...,x_r\}$ is filled with $+$, then $x_\ell,...,x_r$ will remain at $+$. Here $\mathcal{U}$ is supercritical; we assume there exists an update rule $X \in \mathcal{U}$ so that $X \subset \{1,2,...\}$ (the case $X \subset \{...,-2,-1\}$ is similar). Let $x \in \mathds{Z}$. If $x_\ell,...,x_r$ are at $+$ at some time, they (and thus $x$) have fixated at $+$ at this time, since for $x'\in \{x_\ell,...,x_r\}$, $X'\in\mathcal{U}$, all sites of $x'+X'$ are at $+$. We now prove that almost surely there will be some time at which $x_\ell,...,x_r$ are at $+$. For any $\ell \in \mathds{N}^*$, we set $\mathcal{G}_\ell$ the event ``in the time interval $(\ell-1,\ell]$, there are successive clock rings at each site among $x_r,x_{r}-1,...,x_\ell$ that is at $-$ at time $\ell-1$, in increasing order, the update rule chosen is always $X$, and no other clock ring happens in $\{x_\ell,...,x_r\}$''. If $\mathcal{G}_\ell$ occurs, $x_\ell,...,x_r$ are at $+$ at time $\ell$. Moreover, as in the proof of Theorem \ref{thm_connected_components_blocks}, there exists $\varepsilon>0$ depending on $x_\ell,x_r$ so that $\mathds{P}(\mathcal{G}_\ell|\mathcal{F}_{\ell-1}) \geq \varepsilon$. Hence $\mathds{P}(\bigcap_{\ell=1}^n\mathcal{G}_\ell^c) \leq \mathds{E}((1-\varepsilon)^n)$ which tends to 0 when $n$ tends to $+\infty$ by dominated convergence, so almost surely some $\mathcal{G}_\ell$ occurs, then $x_\ell,...,x_r$ are at $+$ at time $\ell$, therefore $x$ fixates almost surely.
\end{proof}

\begin{proof}[Proof of Theorem \ref{thm_flippers_dim1}.]
 The argument is similar to and simpler than the one in the proof of Theorem \ref{thm_flippers}, with $\{x_\ell,...,x_r\}$ playing the role of $\bigcup_{B \in \mathcal{C}}B$ and $\{x_\ell-r,...,x_\ell-1\} \bigcup \{x_r+1,...,x_r+r\}$ playing the role of $\bigcup_{B \in \mathcal{C}^\rightarrow}B$.
\end{proof}

\end{document}